\documentclass[11pt,a4paper]{article}

\usepackage{exscale}
\usepackage{amsmath,amssymb,bbm}
\usepackage{theorem,ifthen}
\usepackage{tikz}

\newtheorem{prop}{Proposition}[section]
\newtheorem{lemma}[prop]{Lemma}
\newtheorem{theo}[prop]{Theorem}
\newtheorem{coroll}[prop]{Corollary}
{\theorembodyfont{\rmfamily}
 \newtheorem{remark}[prop]{Remark}
 \newtheorem{definition}[prop]{Definition}
 \newtheorem{example}[prop]{Example}
}

\newenvironment{proof}[1][Proof]{%
  \begin{list}{}{%
      \settowidth{\labelwidth}{\textit{#1.}}%
      \setlength{\itemindent}{\labelwidth}%
      \addtolength{\itemindent}{\labelsep}%
      \setlength{\leftmargin}{0pt}
      \setlength{\parsep}{0pt}
      \setlength{\listparindent}{\parindent}
    }\item[\textit{#1.}]}
  {\hspace*{\fill}$\Box$ \end{list}}
\newenvironment{proof*}[1][Proof]{%
  \begin{list}{}{%
      \settowidth{\labelwidth}{\textit{#1.}}%
      \setlength{\itemindent}{\labelwidth}%
      \addtolength{\itemindent}{\labelsep}%
      \setlength{\leftmargin}{0pt}
    }\item[\textit{#1.}]}
  {\end{list}}
\newenvironment{definition*}{\begin{definition}}
  {\hspace*{\fill}$\lrcorner$ \end{definition}}
\newenvironment{remark*}{\begin{emark}}
  {\hspace*{\fill}$\lrcorner$ \end{remark}}
\newenvironment{example*}{\begin{example}}
  {\hspace*{\fill}$\lrcorner$ \end{example}}

\newcommand{\braket}[2]{[#1|#2]}
\newcommand{\Braket}[2]{\Bigl[#1\Big|#2\Bigr]}
\newcommand{\Jorth}{{[\perp]}}

\newcommand{\pmat}[1]{\begin{pmatrix}#1\end{pmatrix}}
\newcommand{\set}[2]{\left\{#1\,\middle|\,#2\right\}}
\newcommand{\bigset}[2]{\bigl\{#1\,\big|\,#2\bigr\}}

\newcommand{\iprod}[2]{(#1|#2)}

\newcommand{\iprods}[3]{\iprod{#1}{#2}_{#3,-#3}}
\newcommand{\iprodsr}[3]{\iprod{#1}{#2}_{-#3,#3}}
\newcommand{\ads}{^{(*)}}
\newcommand{\smallmat}[1]{\left(\begin{smallmatrix}#1\end{smallmatrix}\right)}

\DeclareMathOperator{\Real}{Re}

\newcommand{\proj}{\mathrm{pr}}

\newcommand{\R}{\mathbbm{R}}
\newcommand{\C}{\mathbbm{C}}

\newcommand{\mdef}{\mathcal{D}}
\newcommand{\range}{\mathcal{R}}
\newcommand{\eps}{\varepsilon}

\begin{document}

\title{Dichotomous Hamiltonians and
  Riccati equations for systems with unbounded
  control and observation operators
}
\author{Christian Wyss\footnote{
    University of Wuppertal,
    School of
    Mathematics and Natural Sciences,
    Gau\ss stra\ss e 20,
    D-42097 Wuppertal,
    Germany,
    \texttt{wyss@math.uni-wuppertal.de}}}
\date{\today}
\maketitle

\begin{center}
  \parbox{11cm}{\small \textbf{Abstract.}  
    The control algebraic Riccati
    equation is studied for a class of systems with
    unbounded control and observation operators.
    Using a dichotomy property of the associated
    Hamiltonian operator matrix, two
    invariant graph subspaces are constructed which yield
    a nonnegative and a nonpositive solution of the Riccati equation.
    The boundedness of the nonnegative solution and the
    exponential stability of the associated feedback system is proved
    for the case that the
    generator of the system has a compact resolvent.
  } \\[1.5ex]
  \parbox{11cm}{\small\textbf{Keywords.} algebraic Riccati equation, Hamiltonian matrix,
    dichotomous operator, invariant subspace, graph subspace.
  } \\[1.5ex]
  \parbox{11cm}{\small\textbf{Mathematics Subject Classification.}
    Primary 47N70; Secondary 47A15, 47A62, 47B44.
  }
\end{center}

\section{Introduction}

In systems theory,
the algebraic Riccati equation
\begin{equation}\label{eq:care}
  A^*X+XA-XBB^*X+C^*C=0
\end{equation}
plays an important role in many areas.
One example is 
the problem of linear quadratic optimal
control where
a selfadjoint nonnegative solution
is of particular interest.
For infinite-dimensional systems
such a solution is often constructed in parallel
to a solution of the optimal control problem.
This has been done for different kinds of linear systems,
e.g.\ in
\cite{curtain-zwart,opmeer-curtain04,pritchard-salamon,staffans05,weiss-weiss}.

On the other hand, the Riccati equation is closely connected to the
so-called Hamiltonian operator matrix
\begin{equation}
  T=\pmat{A&-BB^*\\-C^*C&-A^*}.
\end{equation}
An operator $X$ is a solution of \eqref{eq:care} if and only if its
associated graph
$\range\smallmat{I\\X}$
is an invariant subspace of the Hamiltonian.
In the finite-dimensional case, this connection has lead to a complete characterisation
of all solutions of the Riccati equation,
see e.g.\ \cite{bittanti-laub-willems,lancaster-rodman} and the references therein.
For infinite-dimensional linear systems, this ``Hamiltonian approach'' to the
Riccati equation has been studied under different
boundedness assumptions on the control and observation operators $B,C$
and for different classes of Hamiltonians concerning their spectral
properties.
For the case that $B,C$ are bounded and have finite rank,
a characterisation of all nonnegative solutions of \eqref{eq:care}
has been obtained in \cite{callier-dumortier-winkin}.
In \cite{kuiper-zwart}
the class of Hamiltonians possessing a Riesz basis of eigenvectors
was considered for systems with bounded $B$ and $C$,
and characterisations of solutions and their properties
were obtained.
In \cite{wyss-rinvsubham,wyss-unbctrlham} this was extended to unbounded $B,C$
and to more general kinds of Riesz bases.
The Riesz basis setting typically leads to the existence of an
infinite number of solutions of \eqref{eq:care}.

However, the existence of a Riesz basis of eigenvectors of $T$ is a strong assumption
and might be to restrictive.
An often weaker condition is that $T$ is \emph{dichotomous}. This means
that the spectrum of $T$ does not contain points
in a strip around the imaginary axis
and that there exist invariant subspaces corresponding to the parts of the
spectrum in the left and right half-plane, respectively.
Dichotomous Hamiltonians with bounded $B$ and $C$ were considered
in \cite{bubak-mee-ran,langer-ran-rotten} and the existence of a nonnegative and a nonpositive
solution of \eqref{eq:care} was shown.
This result was extended in \cite{tretter-wyss} to a setting where
$BB^*$ and $C^*C$ are unbounded closed
operators acting on the state space.
This however excludes PDE systems with control or observation
on the boundary. 
In this article we will construct a nonnegative and a nonpositive solution of \eqref{eq:care}
for a class of dichotomous Hamiltonians which allows for systems
with boundary control and observation.

In the infinite-dimensional setting the Hamiltonian approach typically leads to
unbounded solutions of the Riccati equation in the first instance, see
\cite{langer-ran-rotten,tretter-wyss,wyss-rinvsubham,wyss-unbctrlham}.
This means that the boundedness of solutions is an additional question now.
Moreover, due to the unboundedness of the operators in \eqref{eq:care},
additional care has to be taken to exactly determine the domain
on which the Riccati equation actually holds.

Our setting is as follows:
Let $H,U,Y$ be Hilbert spaces.
Let $A$ be a \emph{quasi-sectorial} operator on $H$, i.e.,
$A-\mu$ is sectorial for some $\mu\geq0$. This means that $A$ may have spectrum on
and to the right of the imaginary axis up to the line $\Real z=\mu$
and that $A$ generates an analytic semigroup.
The operator $A$ determines two scales of Hilbert spaces
$\{H_s\}$ and $\{H_s\ads\}$,
\[H_s\subset H\subset H_{-s},\qquad
H_s\ads\subset H\subset H_{-s}\ads,\qquad s>0,\]
whose norms are given by $\|x\|_s=\|(I+AA^*)^\frac{s}{2}x\|$ and
$\|x\|_s\ads=\|(I+A^*A)^\frac{s}{2}x\|$.
If $A$ is a normal operator, then both scales coincide with the usual
fractional power spaces, $H_s=H_s\ads=\mdef(|A|^s)$.
In general however, the two scales are different and must be distinguished.
Our assumption on the control and observation operators is now
\[B\in L(U,H_{-r}),\qquad C\in L(H_s\ads,Y)\]
where $r,s\geq0$ and $r+s<1$.
Examples of systems with boundary control and observation
which fit into this setting may be found e.g.\ in
\cite{tucsnak-weiss,wyss-unbctrlham}.
The adjoints of $B$ and $C$ are  defined
using a duality relation in each of the scales of Hilbert spaces,
which is induced by the
inner product $\iprod{\cdot}{\cdot}$ on $H$:
the mapping $y\mapsto\iprod{\cdot}{y}$, $y\in H$, extends
by continuity to  isometric isomorphisms $H_{-r}\to (H_r)'$ and
$H_{-s}\ads\to(H_s\ads)'$.
This is also referred to as duality with respect to the pivot space $H$.
With this duality we obtain
\[BB^*:H_r\to H_{-r},\qquad C^*C:H_s\ads\to H_{-s}\ads.\]
The Hamiltonian is now considered as an unbounded operator
\[T_0=\pmat{A&-BB^*\\-C^*C&-A^*}\]
acting on
$V_0=H_{-r}\times H_{-s}\ads$, with  appropriate extensions of the
operators $A$ and $A^*$.
We prove that if
\begin{enumerate}
\item $\sigma(A)\cap i\R=\varnothing$, or
\item $A$ has a compact resolvent
  and
  \[\ker(A-it)\cap\ker C=\ker(A^*+it)\cap\ker B^*=\{0\},\qquad t\in\R,\]
\end{enumerate}
then $T_0$ is dichotomous and hence there is a decomposition $V_0=V_{0+}\oplus V_{0-}$
into $T_0$-invariant subspaces such that $\sigma(T_0|_{V_{0\pm}})\subset\C_\pm$,
i.e., $V_{0-}$ corresponds to the spectrum in the open left half-plane $\C_-$
and $V_{0+}$ to the one in the open right half-plane $\C_+$.
For the rest of this introduction we assume that (a) or (b) is satisfied.

We derive that $V_{0\pm}$ are graph subspaces in two different situations.
In the first we assume that 
\begin{equation}\label{eq:intro-ac}
  \bigcap_{\lambda\in i\R\cap\varrho(A^*)}\ker B^*(A^*-\lambda)^{-1}=\{0\}.
\end{equation}
Then $V_{0\pm}$ are graphs, $V_{0\pm}=\range\smallmat{I\\X_{0\pm}}$, of closed,
possibly unbounded operators
$X_{0\pm}:\mdef(X_{0\pm})\subset H_{-r}\to H_{-s}\ads$.
If in addition
\begin{equation}\label{eq:intro-ao}
  \bigcap_{\lambda\in i\R\cap\varrho(A)}\ker C(A-\lambda)^{-1}=\{0\},
\end{equation}
then $X_{0\pm}$ are also injective and hence $V_{0\pm}=\range\smallmat{Y_{0\pm}\\I}$
with $Y_{0\pm}=X_{0\pm}^{-1}$.
The conditions \eqref{eq:intro-ac} and \eqref{eq:intro-ao} were also used
in \cite{langer-ran-rotten,tretter-wyss,wyss-rinvsubham,wyss-unbctrlham},
sometimes in different but equivalent forms;
\eqref{eq:intro-ac} amounts to the approximate controllability,
\eqref{eq:intro-ao} to the approximate observability of the system $(A,B,C)$,
see \cite{langer-ran-rotten,wyss-unbctrlham}.
In the second situation, we assume that $\sigma(A)\subset\C_-$.
Hence
the semigroup generated by $A$ is exponentially stable. In this case we obtain
$V_{0-}=\range\smallmat{I\\X_{0-}}$ and $V_{0+}=\range\smallmat{Y_{0+}\\I}$
where, again, $X_{0-}$ and $Y_{0+}$ are closed and possibly unbounded, but not
necessarily injective.

Under the additional assumption that $A$ has a compact resolvent,
we can show that $X_{0-}$ and $Y_{0+}$ are bounded.
More precisely, if $A$ has a compact resolvent
and either
\eqref{eq:intro-ac} and \eqref{eq:intro-ao} or $\sigma(A)\subset\C_-$ hold, then
$X_{0-}\in L(H_{-r},H_{-s}\ads)$, $Y_{0+}\in L(H_{-s}\ads,H_{-r})$.
In this case we also obtain that $X_{0-}$ is a solution of the Riccati equation
on the domain $H_{1-r}\ads$ 
and that the operator $A-BB^*X_{0-}$ associated with the closed loop system
generates an exponentially stable semigroup on $H_{-r}$.

In \cite{langer-ran-rotten,tretter-wyss} the two  solutions
of the Riccati equation are selfadjoint operators on $H$, one being nonnegative,
the other nonpositive.
Here the situation is more involved.
While $X_{0\pm}$ can be restricted to symmetric operators on $H$ that are nonnegative
and nonpositive, respectively,
selfadjoint restrictions need not exist in general.
More specifically,
$X_{0\pm}$ admit restrictions to closed operators $X_{1\pm}$ from $H_s\ads$
to $H_r$ such that
\[X_{1\pm}\subset X_{1\pm}^*=X_{0\pm},\]
where the adjoint is computed with respect to the duality in
the scales $\{H_s\}$ and $\{H_s\ads\}$.
In particular, $X_{1\pm}$ is symmetric when considered as an operator
on $H$. If  $X_{M\pm}$ is the closure of $X_{1\pm}$ as an
operator on $H$ and $X_\pm$ is the part of $X_{0\pm}$ in $H$,
then
\[X_{1\pm}\subset X_{M\pm}\subset X_{M\pm}^*=X_\pm\subset X_{0\pm},\]
$X_{M-}$ is symmetric and nonnegative, $X_{M+}$ is symmetric and nonpositive.
We can also consider the restriction of 
the Hamiltonian $T_0$ to an operator $T$ on $V=H\times H$.
Then $T$ has invariant subspaces
$V_\pm$ corresponding to the spectrum in $\C_\pm$
and $V_\pm$ is in fact the graph of $X_\pm$.
Note here that $T$ will in general not be dichotomous
since
$V_+\oplus V_-$ will only be dense in $V$.
Also note that the
above statements hold for $X_{0-}$ and its restrictions
provided that $V_{0-}=\range\smallmat{I\\X_{0-}}$,
i.e., if \eqref{eq:intro-ac} or $\sigma(A)\subset\C_-$ holds.
Likewise the statements for the restrictions of $X_{0+}$ hold
if $V_{0+}=\range\smallmat{I\\X_{0+}}$,
i.e., if \eqref{eq:intro-ac} is true.

Finally assume that $\max\{r,s\}<\frac12$.
In this case $T$ is in fact dichotomous and we obtain $X_{M\pm}=X_\pm$.
Hence
$X_-$ is selfadjoint nonnegative, $X_+$ is selfadjoint nonpositive.
If in addition $A$ has a compact resolvent, then $X_-$ is also bounded
and a restriction of $A-BB^*X_{0-}$ generates
an exponentially stable semigroup on $H$.

This article is organised as follows:
In section~\ref{sec:prel} we collect some general operator theoretic statements,
in particular about dichotomous, sectorial and bisectorial operators.
The scales of Hilbert spaces are defined in section~\ref{sec:scale}
and their basic properties are recalled, in particular concerning interpolation.
Section~\ref{sec:ham} contains the definition of the Hamiltonian and basic facts
about its spectrum.
Moreover we describe the symmetry of the Hamiltonian with respect to
two indefinite inner products, which will be essential in sections~\ref{sec:angular}
and~\ref{sec:sym}.
In section~\ref{sec:bisecham} we prove the bisectoriality and dichotomy of $T_0$
and $T$ using interpolation in the Hilbert scales.
The graph subspace properties of $V_{0\pm}$ and $V_\pm$ are derived in
section~\ref{sec:angular} as well as the boundedness of $X_{0-}$ and $Y_{0+}$.
The symmetry relations between $X_{0\pm}$ and its restrictions are the subject
of section~\ref{sec:sym}, while the Riccati equation and the closed loop operator
are studied in section~\ref{sec:ricc}.

A few remarks on the notation:
We denote the domain of a linear operator $T$ by $\mdef(T)$, its range by $\range(T)$,
the spectrum by $\sigma(T)$ and the resolvent set by $\varrho(T)$.
The space of all bounded linear operators mapping a Banach space $V$ to another
Banach space $W$ is denoted by $L(V,W)$.
For the operator norm of $T\in L(V,W)$ we occasionally write
$\|T\|_{V\to W}$ to make the dependence on the spaces $V$ and $W$ explicit.

\section{Preliminaries}\label{sec:prel}

In this section,
we summarise some concepts and results for linear operators on Banach spaces.
Unless stated explicitly, linear operators are
not assumed to be densely defined.

\begin{lemma}\label{lem:op1}
  Let $T$ be a linear operator on a Banach space $V$. Let $W$ be another Banach space
  such that $\mdef(T)\subset W\subset V$ and such that the imbedding $W\hookrightarrow V$
  is continuous.
  Let $\lambda\in\varrho(T)$.
  \begin{enumerate}
  \item  The resolvent $(T-\lambda)^{-1}$ yields a bounded operator from $V$
    into $W$, i.e.,
    $(T-\lambda)^{-1}\in L(V,W)$.
  \item If the imbedding $W\hookrightarrow V$ is compact,
    then the resolvent is compact as an operator from $V$ into $V$, i.e.,
    $(T-\lambda)^{-1}:V\to V$ is compact.
  \end{enumerate}
\end{lemma}

\begin{proof}
  \begin{enumerate}
  \item
    The assumption $\mdef(T)\subset W$ implies that $(T-\lambda)^{-1}$ maps $V$
    into $W$. The operator $(T-\lambda)^{-1}:V\to W$ is thus well defined, and
    by the closed graph theorem it suffices to show that it is closed.
    Let $x_n\in V$ with $x_n\to x$ in $V$ and $(T-\lambda)^{-1}x_n\to y$ in $W$
    as $n\to\infty$.
    Then $(T-\lambda)^{-1}x_n\to y$ in $V$ by the continuity of the imbedding
    $W\hookrightarrow V$, and also $(T-\lambda)^{-1}x_n\to (T-\lambda)^{-1}x$ in $V$
    since the resolvent is a bounded operator on $V$.
    Consequently $(T-\lambda)^{-1}x=y$ and hence
    $(T-\lambda)^{-1}:V\to W$ is closed.
  \item
    This follows immediately from (a) by composing the bounded operator
    $(T-\lambda)^{-1}:V\to W$ with the
    compact imbedding $W\hookrightarrow V$.
  \end{enumerate}
\end{proof}

\begin{lemma}\label{lem:op2}
  Let $T_0$ be a linear operator on
  a Banach space $V_0$.
  Let $V$ be another Banach space satisfying
  $\mdef(T_0)\subset V\subset V_0$ with continuous imbedding $V\hookrightarrow V_0$.
  Let $T$ be the part of $T_0$ in $V$, i.e., $T$ is the restriction of $T_0$ to the
  domain
  \[\mdef(T)=\set{x\in\mdef(T_0)}{T_0x\in V},\]
  considered as an operator $T:\mdef(T)\subset V\to V$.
  Then
  \begin{enumerate}
  \item $\sigma_p(T)=\sigma_p(T_0)$,
  \item $\varrho(T_0)\subset \varrho(T)$ and
    $(T-\lambda)^{-1}=(T_0-\lambda)^{-1}|_V$ for all $\lambda\in\varrho(T_0)$,
  \item if $\mdef(T_0)$ is dense in $V$, $V$ is dense in $V_0$ and $\varrho(T_0)\neq\varnothing$,
    then $T$ is densely defined.
  \end{enumerate}
\end{lemma}

\begin{proof}
  \begin{enumerate}
  \item This is clear, since $\mdef(T_0)\subset V$ implies that all eigenvectors
    of $T_0$ belong to $V$.
  \item
    Let $\lambda\in\varrho(T_0)$. Then $T-\lambda:\mdef(T)\to V$ is injective
    as a restriction of $T_0-\lambda$.
    Let $y\in V$ and set $x=(T_0-\lambda)^{-1}y$.
    Then $x\in\mdef(T_0)$, which implies $x\in V$ and $T_0x=\lambda x+y\in V$.
    Therefore $x\in\mdef(T)$ and $(T-\lambda)x=y$.
    Hence $T-\lambda:\mdef(T)\to V$ is bijective with inverse
    $(T-\lambda)^{-1}=(T_0-\lambda)^{-1}|_V$.
    Since $(T_0-\lambda)^{-1}\in L(V_0,V)$ by Lemma~\ref{lem:op1}
    and since $V\hookrightarrow V_0$ is continuous, we obtain
    $(T-\lambda)^{-1}\in L(V)$ and thus $\lambda\in\varrho(T)$.
  \item
    Let $\lambda\in\varrho(T_0)$.
    Since $(T_0-\lambda)^{-1}\in L(V_0,V)$ and since $V\subset V_0$ is dense, we get that
    $\mdef(T)=(T_0-\lambda)^{-1}(V)$ is dense in $\mdef(T_0)=(T_0-\lambda)^{-1}(V_0)$
    with respect to the norm in $V$. As $\mdef(T_0)\subset V$ is dense, we
    conclude that $\mdef(T)\subset V$ is dense.
  \end{enumerate}
\end{proof}

Let us recall the definitions and basic properties of sectorial, bisectorial and
dichotomous operators.
For more details we refer the reader to
\cite{engel-nagel,haase,winklmeier-wyss15}.
We denote by
\begin{equation}
  \Sigma_{\frac\pi2+\theta}=\set{\lambda\in\C\setminus\{0\}}{
  \arg\lambda\in\Bigl[-\frac\pi2-\theta,\frac\pi2+\theta\Bigr]}
\end{equation}
the sector containing the positive real axis with semi-angle
$\frac\pi2+\theta$. We also consider the corresponding bisector
around the imaginary axis
\begin{equation}\label{eq:bisector}
  \Omega_\theta=\Sigma_{\frac\pi2+\theta}\cap\left(-\Sigma_{\frac\pi2+\theta}\right)
  =\set{\lambda\in\C\setminus\{0\}}{
  |\arg\lambda|\in\Bigl[\frac\pi2-\theta,\frac\pi2+\theta\Bigr]}.
\end{equation}

For sectorial operators we adopt the convention that the spectrum
is contained in a sector in the left half-plane:
\begin{definition}\label{def:sectorial}
  A linear operator $S$ on a Banach space $V$
  is called \emph{sectorial} if there
  exist $\theta\geq0$ and $M>0$ such that $\Sigma_{\frac\pi2+\theta}\subset\varrho(S)$ and
  \begin{equation}\label{eq:sectorial}
    \|(S-\lambda)^{-1}\|\leq\frac{M}{|\lambda|} \qquad \text{for all }\;
    \lambda\in\Sigma_{\frac\pi2+\theta}.
  \end{equation}
  $S$ is called \emph{quasi-sectorial} if $S-\mu$ is sectorial for some $\mu\in\R$.
\end{definition}

If  \eqref{eq:sectorial} holds for some $\theta$, then it also holds
for some $\theta'>\theta$ (with a typically larger constant $M$).
We may therefore always assume that $\theta>0$.
$S$ is quasi-sectorial if and only if there exist $\theta,M,\rho>0$
such that\footnote{
  $B_r(z)\subset\C$ denotes the open disc with radius $r$
  centred at $z$.}
$\Sigma_{\frac\pi2+\theta}\setminus B_\rho(0)\subset\varrho(S)$ and
\begin{equation}\label{eq:quasisect}
  \|(S-\lambda)^{-1}\|\leq\frac{M}{|\lambda|} \qquad\text{for all }\;
  \lambda\in\Sigma_{\frac\pi2+\theta},\;
  |\lambda|\geq \rho.
\end{equation}
An operator is  sectorial and densely defined if and only if it is the
generator of a bounded analytic semigroup.
On reflexive Banach spaces every sectorial and quasi-sectorial operator
is densely defined.
If $S$ is a (quasi-) sectorial operator on a Hilbert space,
then its adjoint $S^*$ is also (quasi-) sectorial with the same constants
$\theta$, $M$ (and $\mu,\rho$).

\begin{definition}
  A linear operator $S$ on $V$ is called \emph{bisectorial}
  if $i\R\setminus\{0\}\subset\varrho(S)$ and
  \begin{equation}\label{eq:bisect}
    \|(S-\lambda)^{-1}\|\leq\frac{M}{|\lambda|} \quad\text{for all }\;
    \lambda\in i\R\setminus\{0\}
  \end{equation}
  with some constant $M>0$. $S$ is \emph{almost bisectorial}
  if $i\R\setminus\{0\}\subset\varrho(S)$ and there exist $0<\beta<1$, $M>0$
  such that
  \begin{equation}\label{eq:almbisect}
    \|(S-\lambda)^{-1}\|\leq\frac{M}{|\lambda|^\beta} \quad\text{for all }\;
    \lambda\in i\R\setminus\{0\}.
  \end{equation}
\end{definition}

If $S$ is bisectorial, then for some $\theta>0$ the bisector $\Omega_\theta$
is contained in the resolvent set $\varrho(S)$, 
and an estimate \eqref{eq:bisect} holds for all $\lambda\in\Omega_\theta$.
Similarly, for an almost bisectorial operator a  parabola shaped region around
the imaginary axis belongs to $\varrho(S)$.
If $S$ is bisectorial and $0\in\varrho(S)$, then $S$ is almost bisectorial too,
for any $0<\beta<1$.
Note that an almost bisectorial operator always satisfies $0\in\varrho(S)$,
while for a bisectorial operator $0\in\sigma(S)$ is possible.
Bisectorial operators on reflexive spaces are always densely defined; for
almost bisectorial operators this need not be the case.

\begin{definition}\label{def:dichot}
  A linear operator $S$ on a Banach space $V$ is called \emph{dichotomous}
  if $i\R\subset\varrho(S)$ and there exist closed $S$-invariant subspaces
  $V_\pm$ of $V$ such that $V=V_+\oplus V_-$ and
  \[\sigma(S|_{V_+})\subset\C_+, \quad \sigma(S|_{V_-})\subset \C_-.\]
  $S$ is \emph{strictly dichotomous} if in addition
  $\|(S|_{V_\pm}-\lambda)^{-1}\|$ is bounded on $\C_\mp$.
\end{definition}

A dichotomous operator is block diagonal with respect to the decomposition
$V=V_+\oplus V_-$, see \cite[Remark~2.3 and Lemma~2.4]{tretter-wyss}.
In particular, $\sigma(S)=\sigma(S|_{V_+})\cup\sigma(S|_{V_-})$ and
the subspaces
$V_\pm$ are also $(S-\lambda)^{-1}$-invariant for all $\lambda\in\varrho(S)$.
The additional condition of strict dichotomy ensures that the
invariant subspaces $V_\pm$ are uniquely determined by the operator.

One of the main results from \cite{winklmeier-wyss15}
is that if the resolvent of an operator $S$ is uniformly bounded along the
imaginary axis, then $S$ possesses invariant subspaces $V_\pm$
having the same properties as in Definition~\ref{def:dichot},
with the exception that $V_+\oplus V_-$ might be a proper subspace of $V$,
i.e., $S$ need not necessarily be dichotomous.
In this case, the corresponding projections are unbounded.
We summarise the results for the almost bisectorial situation here.

Let $S$ be an almost bisectorial operator.
Then there exists $h>0$ such that
\(\set{\lambda\in\C}{|\Real\lambda|\leq h}\subset\varrho(S)\)
and the integrals
\begin{equation}\label{eq:Lpm-def}
  L_\pm=\frac{\pm 1}{2\pi i}\int_{\pm h-i\infty}^{\pm h+i\infty}
  \frac{1}{\lambda}(S-\lambda)^{-1}\,d\lambda
\end{equation}
define bounded operators $L_\pm\in L(V)$ which satisfy
\begin{equation}\label{eq:Lpm-prop}
  L_+L_-=L_-L_+=0,\qquad L_++L_-=S^{-1},
\end{equation}
see \cite[\S5]{winklmeier-wyss15}.

\begin{theo}
  \label{theo:sectdichot}
  Let $S$ be almost bisectorial on the Banach space $V$.
  Then $P_\pm=SL_\pm$ are closed complementary projections,
  the subspaces $V_\pm=\range(P_\pm)$
  are closed, $S$- and $(S-\lambda)^{-1}$-invariant for all $\lambda\in\varrho(S)$,
  and
  \begin{enumerate}
  \item $\sigma(S)=\sigma(S|_{V_+})\cup\sigma(S|_{V_-})$
    with $\sigma(S|_{V_\pm})\subset\C_\pm$,
  \item $\|(S|_{V_\pm}-\lambda)^{-1}\|$ is bounded on $\C_\mp$,
  \item $\mdef(S)\subset\mdef(P_\pm)=V_+\oplus V_-\subset V$,
  \item $I=P_++P_-$ on $\mdef(P_\pm)$.
  \end{enumerate}
  The projections satisfy the identity
  \begin{equation}\label{eq:dichotint2}
    P_+x-P_-x=\frac{1}{\pi i}\int_{-i\infty}^{i\infty\,\prime}
    (S-\lambda)^{-1}x\,d\lambda,
    \qquad x\in\mdef(S),
  \end{equation}
  where the prime denotes the Cauchy principal value at infinity.
  Moreover, $S$ is strictly dichotomous if and only if $P_\pm\in L(V)$.
\end{theo}
\begin{proof}
  All assertions follow from Theorem~4.1 and~5.6
  as well as Corollary~4.2 and~5.9
  in \cite{winklmeier-wyss15}.
\end{proof}  
  
Note that $P_\pm$ are closed complementary projections in the sense that
they are closed operators on $V$ and satisfy
$\range(P_\pm)\subset\mdef(P_\pm)$, $P_\pm^2=P_\pm$,
$\mdef(P_+)=\mdef(P_-)$
and $I=P_++P_-$ on $\mdef(P_\pm)$.
In other words, $P_\pm$ are complementary
projections in the algebraic sense acting on the space $\mdef(P_+)=\mdef(P_-)$. 
Since $S$ is invertible, we obtain
\begin{equation}
  V_\pm=\range(P_\pm)=\ker P_\mp=\ker L_\mp.
\end{equation}
The case that $P_\pm$ are unbounded may occur even for bisectorial
and almost bisectorial $S$,
see Examples~5.8 and~8.2 in \cite{winklmeier-wyss15}.

For use in later sections, we collect some properties of the spaces
$\range(L_\pm)$:
\begin{lemma}\label{lem:Lpm-range}
  Let $S$ be an almost bisectorial operator. Then the inclusions
  \begin{equation}\label{eq:Lpm-range-incl}
      \mdef(S)\cap V_\pm\subset\range(L_\pm) 
  \subset V_\pm
  \end{equation}
  hold, in particular $\overline{\range(L_\pm)}\subset V_\pm$.
  In addition,
  \begin{enumerate}
  \item if $S$ is also densely defined, then
    $\overline{\mdef(S)\cap V_\pm}=\overline{\range(L_\pm)}$;
  \item if $S$ is densely defined and strictly dichotomous, then
    $\mdef(S)\cap V_\pm=\range(L_\pm)$ and $\overline{\range(L_\pm)}=V_\pm$.
  \end{enumerate}
\end{lemma}
\begin{proof}
  From \eqref{eq:Lpm-prop} and the invariance properties of $V_\pm$ we get
  \[\mdef(S)\cap V_\pm=S^{-1}(V_\pm)=L_\pm(V_\pm)\subset
  \range(L_\pm)\subset\ker L_\mp=V_\pm.\]
  Since $V_\pm$ are closed, $\overline{\range(L_\pm)}\subset V_\pm$ follows.
  If $S$ is densely defined, then part (c) of the previous theorem
  yields $\overline{V_+\oplus V_-}=V$.
  Therefore
  \[\range(L_\pm)=L_\pm(\overline{V_+\oplus V_-})\subset
  \overline{L_\pm(V_+\oplus V_-)}=\overline{L_\pm(V_\pm)}
  =\overline{\mdef(S)\cap V_\pm},\]
  and hence the inclusion ``$\supset$'' in (a) holds.
  The other inclusion is clear by \eqref{eq:Lpm-range-incl}.
  If now $S$ is also strictly dichotomous, then $P_\pm$ are bounded.
  In particular $\range(L_\pm)\subset\mdef(S)$ and hence
  $\range(L_\pm)=\mdef(S)\cap V_\pm$.
  Using that $S$ and $L_\pm$ commute, we obtain
  \[V_\pm=\range(P_\pm)=P_\pm(\overline{\mdef(S)})\subset
  \overline{P_\pm(\mdef(S))}=\overline{L_\pm S(\mdef(S))}=\overline{\range(L_\pm)}\]
  and hence $\overline{\range(L_\pm)}=V_\pm$ by \eqref{eq:Lpm-range-incl}.
  %
\end{proof}

We remark that the inclusion
$\overline{\range(L_\pm)}\subset V_\pm$ is strict in general, see
\cite[\S6]{winklmeier-wyss15}
and Examples~8.3 and~8.5 in \cite{winklmeier-wyss15}.

\section{Two scales of Hilbert spaces associated with a closed operator}
\label{sec:scale}

In this section we construct two scales of Hilbert spaces $\{H_s\}$ and $\{H_s\ads\}$
associated with a closed, densely defined operator $A$.
Although the results are well known,
the presentations found in the literature often cover only parts of the
full theory
or are restricted to certain special cases:
The construction of the spaces $H_{\pm1}$ and $H_{\pm1}\ads$ for general $A$
can be found e.g.\ in \cite{jonas-langer,tucsnak-weiss}.
The intermediate spaces for $s=\pm\frac12$ are defined in
\cite{jonas-langer} for general, and in \cite{tucsnak-weiss} for
selfadjoint positive $A$.
The spaces $H_s$ with arbitrary $s$ are constructed in \cite{jonas-trunk}
for selfadjoint $A$,
while a general theory of scales of Hilbert spaces including
interpolation results is contained in \cite{berezanskii}.
Note that in \cite{tucsnak-weiss} a different naming convention and
different but equivalent
definitions of the spaces are used.
Our presentation follows \cite{berezanskii,jonas-langer}.

Let $A$ be a closed, densely defined linear operator on a
separable Hilbert space $H$.
We denote by $\|\cdot\|$ the norm on $H$
and consider the positive selfadjoint operator $\Lambda=(I+AA^*)^\frac12$.
For $s>0$ let $H_s=\mdef(\Lambda^s)$ be equipped with the norm
$\|x\|_s=\|\Lambda^s x\|$, and let $H_{-s}$ be the completion of $H$
with respect to the norm $\|x\|_{-s}=\|\Lambda^{-s}x\|$.
Then $H_s$ and $H_{-s}$ are Hilbert spaces,
\[H_s\subset H\subset H_{-s},\]
and the imbeddings are continuous and dense.
The family of spaces $\{H_s\}$ is called a \emph{scale of Hilbert spaces}.
In particular we obtain $H_1=\mdef(A^*)$ and
\[\|x\|_1=(\|x\|^2+\|A^*x\|^2)^{\frac12}, \quad x\in H_1.\]

For any $s>0$, the spaces $H_s$ and $H_{-s}$ are dual to each other
with respect to the inner product
$\iprod{\cdot}{\cdot}$ of $H$.
More precisely, the norm on $H_s$ satisfies
\[\|y\|_{-s}=\sup\bigset{|\iprod{x}{y}|}{x\in H_s,\,\|x\|_s=1}, \quad y\in H,\]
which implies that the inner product of $H$ extends by continuity
to a bounded sesquilinear form
on $H_s\times H_{-s}$, which we denote by $\iprods{\cdot}{\cdot}{s}$.
In fact,
\[\iprods{x}{y}{s}=\iprod{\Lambda^s x}{\Lambda^{-s} y},
\quad x\in H_s,\,y\in H.\]
The space $H_{-s}$ can now be identified with the dual space of $H_s$ by means
of the isometric isomorphism $H_{-s}\to (H_s)'$, $y\mapsto\iprods{\cdot}{y}{s}$.
For convenience, we also define a sesquilinear form on $H_{-s}\times H_s$
by
\[\iprodsr{y}{x}{s}=\overline{\iprods{x}{y}{s}},\quad x\in H_s,\, y\in H_{-s}.\]

With respect to the duality in the scale $\{H_s\}$ we obtain the following
notion of adjoint operators:

\begin{definition}
  Let $W$ be a Hilbert space and $C\in L(H_s,W)$.
  Then the operator $C^*\in L(W,H_{-s})$ satisfying
  \begin{equation}\label{eq:sadjoint}
    \iprod{Cx}{w}_W=\iprods{x}{C^*w}{s}, \quad x\in H_s,\,w\in W,
  \end{equation}
  where $\iprod{\cdot}{\cdot}_W$ denotes the inner product of $W$,
  is called the \emph{adjoint of $C$ with respect to the scale $\{H_s\}$.}
  Similarly the adjoint of $B\in L(W,H_{-s})$ with respect to $\{H_s\}$
  is the operator $B^*\in L(H_s,W)$
  such that
  \begin{equation}\label{eq:sadjoint2}
    \iprods{x}{Bw}{s}=\iprod{B^*x}{w}_W, \quad x\in H_s,\,w\in W.
  \end{equation}
\end{definition}

The adjoints exist, are uniquely determined and satisfy
$B=B^{**}$, $C=C^{**}$, $\|B\|=\|B^*\|$ and $\|C\|=\|C^*\|$.
The adjoints of $\tilde C\in L(W,H_s)$ and $\tilde B\in L(H_{-s},W)$
are defined in a similar way.
If $C\in L(H_s,W)$ is an isomorphism, then $C^*$ is an isomorphism too
and $(C^*)^{-1}=(C^{-1})^*$.



\begin{remark}
The notion of adjoints with respect to the scale $\{H_s\}$
generalises the usual definition of adjoints of  unbounded
operators on Hilbert spaces:
Let $C\in L(H_s,W)$. Then $C$ can be regarded as a densely defined unbounded
operator $C_1:\mdef(C_1)\subset H\to W$ with domain $\mdef(C_1)=H_s$.
The adjoint of $C_1$ in the usual sense of unbounded operators is an operator
$C_1^*:\mdef(C_1^*)\subset W\to H$.
Observe that $C_1$ and $C_1^*$ satisfy
\eqref{eq:sadjoint} provided that $w\in\mdef(C_1^*)$.
Consequently
$C_1^*$ is a restriction of $C^*:W\to H_{-s}$. In fact
\[\mdef(C_1^*)=\set{w\in W}{C^*w\in H}.\]
Note here that $C\in L(H_s,W)$ does not imply that
$C_1$ is closable. Hence $C_1^*$ need not be densely
defined and even $\mdef(C_1^*)=\{0\}$ is possible.
\end{remark}

Since $H_1=\mdef(A^*)$ and
since $\|\cdot\|_1$ is equal to the graph norm of $A^*$,
we can consider $A^*$ as a bounded operator $A^*:H_1\to H$.
The adjoint with respect to $\{H_s\}$ is a bounded operator $A^{**}:H\to H_{-1}$
and in view of the last remark $A^{**}$ is an extension of the original
operator $A$. We will denote this extension by $A$ again,
\[A:H\to H_{-1}.\]
Now for any $\lambda\in\varrho(A)$, the operator
$A^*-\bar\lambda:H_1\to H$ is an isomorphism. Hence its adjoint
$A-\lambda:H\to H_{-1}$ is an isomorphism too.
In particular  $\|(A-\lambda)^{-1}\cdot\|$
is an equivalent norm on $H_{-1}$.

Consider now the positive selfadjoint operator $\Lambda_*=(I+A^*A)^{\frac12}$,
and let $\{H\ads_s\}$ be the scale of Hilbert spaces associated with it.
In other words, we repeat the above construction with the roles of $A$ and $A^*$ interchanged.
We denote the respective norms 
and the extension of the inner product
by $\|\cdot\|\ads_s$, $\|\cdot\|\ads_{-s}$
and $\iprods{\cdot}{\cdot}{s}\ads$.
Moreover $H\ads_1=\mdef(A)$, the norm on $H\ads_1$ is equal to the graph norm of $A$,
the norm on $H\ads_{-1}$ is equivalent to $\|(A^*-\lambda)^{-1}\cdot\|$ for
$\lambda\in\varrho(A^*)$, and we get bounded operators
\[A:H\ads_1\to H,\qquad A^*:H\to H\ads_{-1}.\]

\begin{lemma}\label{lem:scale-cpt}
  If $A$ has a compact resolvent, then the imbeddings
  $H_s\hookrightarrow H$ and $H_s\ads\hookrightarrow H$
  are compact for all $s>0$.
\end{lemma}
\begin{proof}
  Let $\lambda\in\varrho(A)$. So $(A-\lambda)^{-1}$ and $(A^*-\bar\lambda)^{-1}$
  are compact operators in $L(H)$.
  The imbedding $H_1\hookrightarrow H$ can be written as the composition
  \[H_1\xrightarrow{A^*-\bar\lambda} H \xrightarrow{(A^*-\bar\lambda)^{-1}} H.\]
  Since $A^*-\bar\lambda:H_1\to H$ is bounded, it follows that $H_1\hookrightarrow H$ is compact.
  Since $\Lambda^{-1}:H\to H_1$ is bounded, the sequence
  \[H\xrightarrow{\Lambda^{-1}} H_1\hookrightarrow H\]
  implies that the operator $\Lambda^{-1}:H\to H$ is compact.
  Consequently $\Lambda^{-s}:H\to H$ is also compact for all $s>0$.
  Decomposing $H_s\hookrightarrow H$ as
  \[H_s\xrightarrow{\Lambda^{s}} H\xrightarrow{\Lambda^{-s}} H\]
  where $\Lambda^s:H_s\to H$ is bounded, we conclude that
  $H_s\hookrightarrow H$ is compact.
  The proof for $H_s\ads\hookrightarrow H$ is analogous.
\end{proof}

For operators acting between two scales of Hilbert spaces, there is
the following  interpolation result,  which is also known as Heinz' inequality,
see \cite[Theorem I.7.1]{krein}.
Let $H$ and $G$ be Hilbert spaces.
Consider
the scales of Hilbert spaces  $\{H_s\}$ and $\{G_r\}$
with corresponding
positive selfadjoint operators  $\Lambda$ and $\Delta$ on $H$ and $G$, respectively.
\begin{theo}[\mbox{\cite[Theorem III.6.10]{berezanskii}}]
  Let $r_1<r_2$, $s_1<s_2$ and let $B:G_{r_1}\to H_{s_1}$ be a bounded linear operator
  which restricts to a bounded operator $B:G_{r_2}\to H_{s_2}$.
  Let $0<\lambda<1$ and
  \[r=\lambda r_1+(1-\lambda)r_2,\quad s=\lambda s_1+(1-\lambda)s_2.\]
  Then  $B$ also restricts to a bounded operator $B:G_r\to H_s$ and
  \[\|B\|_{G_r\to H_s}\leq\|B\|^\lambda_{G_{r_1}\to H_{s_1}}\|B\|^{1-\lambda}_{G_{r_2}\to H_{s_2}}.\]
\end{theo}
We remark that if $B$ restricts to an operator $B:G_{r_2}\to H_{s_2}$, i.e.,
if $B$ maps $G_{r_2}$ into $H_{s_2}$, then
the boundedness of the restriction already follows from
the closed graph theorem.

Applying interpolation to $A:H\ads_1\to H$ and its extension
$A:H\to H_{-1}$, we obtain that $A$ also acts as a bounded operator
\[A:H\ads_{1-s}\to H_{-s}, \qquad s\in[0,1].\]
Similarly,
\[A^*:H_{1-s}\to H\ads_{-s}, \qquad s\in[0,1].\]
Moreover, if $\lambda\in\varrho(A)$ then
$A-\lambda:H\ads_{1-s}\to H_{-s}$ and $A^*-\bar\lambda:H_{1-s}\to H\ads_{-s}$
are both isomorphisms.
Here surjectivity follows from the fact that for example
the resolvent $(A-\lambda)^{-1}$ is an operator in $L(H,H\ads_1)$ and $L(H_{-1},H)$
and hence by interpolation also in $L(H_{-s},H\ads_{1-s})$.

The extensions of $A$ and $A^*$ satisfy the identity
\begin{equation}\label{eq:extAadj}
  \iprodsr{Ax}{y}{s}=\iprod{x}{A^*y}_{1-s,s-1}\ads,\quad x\in H_{1-s}\ads,\,y\in H_{s}.
\end{equation}
This follows from an extension by continuity of the relation
$\iprod{Ax}{y}=\iprod{x}{A^*y}$, $x\in\mdef(A)$, $y\in\mdef(A^*)$.

In view of the above, using appropriate restrictions and extensions,
the resolvent $(A-\lambda)^{-1}$
belongs to $L(H)$ as well as $L(H_{-1})$ and $L(H_1\ads)$.
Similarly, $(A^*-\bar\lambda)^{-1}$ belongs to $L(H)$, $L(H_{-1}\ads)$
and $L(H_1)$. The corresponding operator norms can be
estimated as follows:
\begin{lemma}
  For any $\lambda\in\varrho(A)$ the estimates
  \[\|(A-\lambda)^{-1}\|_{L(H_{-1})}\leq \|(A^*-\bar\lambda)^{-1}\|_{L(H_1)}
  \leq \|(A-\lambda)^{-1}\|_{L(H)}\]
  and
  \[\|(A^*-\bar\lambda)^{-1}\|_{L(H_{-1}\ads)}\leq\|(A-\lambda)^{-1}\|_{L(H_1\ads)}
  \leq\|(A^*-\bar\lambda)^{-1}\|_{L(H)}\]
  hold.
\end{lemma}
\begin{proof}
  From
  \begin{align*}
    \|(A^*-\bar\lambda)^{-1}x\|_1^2
    &=\|(A^*-\bar\lambda)^{-1}x\|^2+\|A^*(A^*-\bar\lambda)^{-1}x\|^2\\
    &=\|(A^*-\bar\lambda)^{-1}x\|^2+\|(A^*-\bar\lambda)^{-1}A^*x\|^2\\
    &\leq\|(A^*-\bar\lambda)^{-1}\|_{L(H)}^2\|x\|_1^2
  \end{align*}
  for $x\in H_1$
  we obtain
  \[\|(A^*-\bar\lambda)^{-1}\|_{L(H_1)}\leq\|(A^*-\bar\lambda)^{-1}\|_{L(H)}
  =\|(A-\lambda)^{-1}\|_{L(H)}.\]
  Moreover for $x\in H_1$, $y\in H_{-1}$,
  \begin{align*}
    |\iprod{x}{(A-\lambda)^{-1}y}|
    &=|\iprods{(A^*-\bar\lambda)^{-1}x}{y}{1}|\\
    &\leq\|(A^*-\bar\lambda)^{-1}\|_{L(H_1)}\|x\|_1\|y\|_{-1},
  \end{align*}
  which implies
  $\|(A-\lambda)^{-1}y\|_{-1}
  \leq\|(A^*-\bar\lambda)^{-1}\|_{L(H_1)}\|y\|_{-1}$
  and hence
  \[\|(A-\lambda)^{-1}\|_{L(H_{-1})}\leq\|(A^*-\bar\lambda)^{-1}\|_{L(H_1)}.\]
  The other estimates are analogous.
\end{proof}

Interpolation now yields the following:
\begin{coroll}\label{cor:scale-Anorms}
  For $\lambda\in\varrho(A)$, $s\in[0,1]$,
  \begin{gather*}
    \|(A-\lambda)^{-1}\|_{L(H_{-s})}\leq\|(A-\lambda)^{-1}\|_{L(H)},\\
    \|(A^*-\bar\lambda)^{-1}\|_{L(H_{-s}\ads)}\leq\|(A^*-\bar\lambda)^{-1}\|_{L(H)}.
  \end{gather*}
\end{coroll}

\section{The Hamiltonian}
\label{sec:ham}

Let $A$ be a closed, densely defined operator on a Hilbert space $H$
and let $\{H_s\}$ and $\{H\ads_s\}$ be the associated scales of
Hilbert spaces defined in Section~\ref{sec:scale}.
Let
\[B\in L(U, H_{-r}), \qquad C\in L(H\ads_{s},Y)\]
where
$U,Y$ are additional Hilbert spaces and $r,s\in[0,1]$ satisfy $r+s\leq 1$.
The adjoints of $B$ and $C$ with respect to the scales of Hilbert spaces
are
\[B^*\in L(H_r, U), \qquad C^*\in L(Y,H\ads_{-s}).\]
We define the \emph{Hamiltonian} as the operator matrix
\[T_0=\pmat{A&-BB^*\\-C^*C&-A^*}.\]
Then $T_0$ is a well-defined linear operator from
\(\mdef(T_0)=H\ads_{1-r}\times H_{1-s}\)
to the product Hilbert space
\[V_0=H_{-r}\times H\ads_{-s}.\]
Indeed we have
\begin{gather*}
  A:H\ads_{1-r}\to H_{-r},  \qquad
  BB^*:H_r\to H_{-r},\\
  C^*C:H\ads_s\to H\ads_{-s},
  \qquad A^*:H_{1-s}\to H\ads_{-s},
\end{gather*}
and the assumption $r+s\leq1$ implies 
\[H\ads_{1-r}\subset H\ads_s, \qquad H_{1-s}\subset H_r.\]
We consider $T_0$ as an unbounded operator on $V_0$ with domain $\mdef(T_0)$ as above.
In particular, $T_0$ is densely defined.

Alongside $V_0$ we will also consider the two product Hilbert spaces
\[V_1=H_s\ads\times H_r \qquad\text{and}\qquad V=H\times H.\]
Thus
\[\mdef(T_0)\subset V_1\subset V\subset V_0.\]
Let $T$ be the part of $T_0$ in $V$.
Then $\sigma_p(T)=\sigma_p(T_0)$.
Moreover  $T$ will be densely defined as soon as
$\varrho(T_0)\neq\varnothing$.
This follows from Lemma~\ref{lem:op2}
since both inclusions $\mdef(T_0)\subset V$ and $V\subset V_0$ are dense.

\begin{lemma}\label{lem:pointgap}
  The Hamiltonian satisfies
  \[\sigma_p(T_0)\cap i\R=\varnothing\]
  if and only if
  \begin{equation}\label{eq:pointgapcond}
    \ker(A-it)\cap\ker C=\ker(A^*+it)\cap\ker B^*=\{0\}
    \quad\text{for all}\quad t\in\R.
  \end{equation}
\end{lemma}

\begin{proof}
  Suppose first that \eqref{eq:pointgapcond} holds and that
  \[\pmat{x\\y}\in\mdef(T_0),\qquad T_0\pmat{x\\y}=it\pmat{x\\y},\qquad t\in\R.\]
  Then
  \[(A-it)x-BB^*y=0,\qquad -C^*Cx-(A^*+it)y=0\]
  where $x\in H_{1-r}\ads\subset H_s\ads$, $y\in H_{1-s}\subset H_r$.
  Using the extended inner products of the scales $\{H_s\}$ and $\{H\ads_s\}$,
  we find
  \begin{equation}\label{eq:gapproof1}
    \begin{aligned}
    0&=\iprod{(A-it)x-BB^*y}{y}=\iprodsr{(A-it)x}{y}{r}-\iprodsr{BB^*y}{y}{r},\\
    0&=\iprod{-C^*Cx-(A^*+it)y}{x}=-\iprodsr{C^*Cx}{x}{s}\ads-\iprodsr{(A^*+it)y}{x}{s}\ads.
    \end{aligned}
  \end{equation}
  From \eqref{eq:extAadj} we see that
  \[\iprodsr{Ax}{y}{r}=\iprods{x}{A^*y}{s}\ads,\quad x\in H_{1-r}\ads,\,y\in H_{1-s}.\]
  Adding the two equations in \eqref{eq:gapproof1} and taking the real part,
  we thus obtain
  \[0=-\iprodsr{BB^*y}{y}{r}-\iprodsr{C^*Cx}{x}{s}\ads=
  -\|B^*y\|_U^2-\|Cx\|_Y^2.\]
  Consequently $B^*y=Cx=0$ and hence also $(A-it)x=(A^*+it)y=0$.
  Now \eqref{eq:pointgapcond} implies $x=y=0$ and so $it\not\in\sigma_p(T_0)$.
  For the reverse implication note that if for example
  $x\in\ker(A-it)\cap\ker C$ and $x\neq0$, then $(x,0)$ is an eigenvector of $T_0$
  with eigenvalue $it$.
\end{proof}

\begin{lemma}\label{lem:apprgap}
  The Hamiltonian satisfies
  \begin{equation}\label{eq:spectralgapincl}
    \sigma_\mathrm{app}(T_0)\cap i\R\subset\sigma(A).
  \end{equation}
\end{lemma}

\begin{proof}
  Let $t\in\R$, $it\in\sigma_\mathrm{app}(T_0)$.
  Then there exist $v_n\in\mdef(T_0)$ such that
  $\|v_n\|_{V_0}=1$
  and
  \[\lim_{n\to\infty}(T_0-it)v_n=0 \quad\text{in}\quad V_0.\]
  By the continuity of the imbedding $V_1\hookrightarrow V_0$
  there is a constant $c>0$ such that
  \[1=\|v_n\|_{V_0}\leq c\|v_n\|_{V_1}.\]
  Thus also
  \[\lim_{n\to\infty}(T_0-it)\frac{v_n}{\|v_n\|_{V_1}}=0 \quad\text{in}\quad V_0.\]
  Setting
  $(x_n,y_n)=v_n/\|v_n\|_{V_1}$
  we obtain
  $\|x_n\|_s^{(*)2}+\|y_n\|_r^2=1$ and
  \[\lim_{n\to\infty}(T_0-it)\pmat{x_n\\y_n}=0 \quad\text{in}\quad V_0,\]
  or
  \begin{equation}\label{eq:gapproof2}
    \begin{alignedat}{2}
    (A-it)x_n-BB^*y_n&\to0 \quad&&\text{in}\quad H_{-r}\\
    -C^*Cx_n-(A^*+it)y_n&\to0 &&\text{in}\quad H_{-s}\ads
    \end{alignedat}
  \end{equation}
  as $n\to\infty$.
  Since the sequences $(x_n)$ and $(y_n)$ are bounded in $H_s\ads$ and $H_r$, respectively,
  this implies that
  \begin{align*}
    &\iprodsr{(A-it)x_n-BB^*y_n}{y_n}{r}\to0\\
    &\iprodsr{-C^*Cx_n-(A^*+it)y_n}{x_n}{s}\ads\to0
  \end{align*}
  Similarly to the previous proof, we add these identities and take the real part to obtain
  \[-\iprodsr{BB^*y_n}{y_n}{r}-\iprodsr{C^*Cx_n}{x_n}{s}\ads
  =-\|B^*y_n\|_U^2-\|Cx_n\|_Y^2\to0.\]
  Consequently $B^*y_n\to0$ and $Cx_n\to0$.

  Now suppose in addition that $it\in\varrho(A)$.
  Then $A-it$ is an isomorphism from $H_{1-r}\ads$ to $H_{-r}$,
  see section~\ref{sec:scale}.
  Therefore $(A-it)^{-1}\in L(H_{-r},H_s\ads)$ and analogously
  $(A^*+it)^{-1}\in L(H_{-s}\ads,H_r)$.
  It follows that
  \[(A-it)^{-1}BB^*y_n\to0 \quad\text{in}\quad H_s\ads, \qquad
  (A^*+it)^{-1}C^*Cx_n\to 0 \quad\text{in}\quad H_r.\]
  On the other hand, we infer from \eqref{eq:gapproof2} that
  \begin{alignat*}{2}
    x_n-(A-it)^{-1}BB^*y_n&\to0 \quad&&\text{in}\quad H_s\ads,\\
    -(A^*+it)^{-1}C^*Cx_n-y_n&\to0 &&\text{in}\quad H_r.
  \end{alignat*}
  Therefore $x_n\to0$ in $H_s\ads$ and $y_n\to0$ in $H_r$, which
  contradicts $\|x_n\|_s^{(*)2}+\|y_n\|_r^2=1$.
\end{proof}

\begin{lemma}\label{lem:ham-cptresolv}
  If $A$ has a compact resolvent, $r+s<1$ and $\varrho(T_0)\neq\varnothing$,
  then both $T$ and $T_0$ have a compact resolvent too.
\end{lemma}
\begin{proof}
  First we have $\varrho(T)\neq\varnothing$ by Lemma~\ref{lem:op2}.
  Lemma~\ref{lem:scale-cpt} shows that the imbeddings
  $H_{1-r}\ads\times H_{1-s}\hookrightarrow V$
  and $H_{1-r}\ads\times H_{1-s}\hookrightarrow V_0$ are
  compact. 
  Since $\mdef(T)\subset\mdef(T_0)=H_{1-r}\ads\times H_{1-s}$, Lemma~\ref{lem:op1}
  implies that the resolvents of $T$ and $T_0$ are compact.
\end{proof}

On $V=H\times H$ we consider the two indefinite inner products
\begin{equation}\label{eq:Jinner}
  \braket{v}{w}=\iprod{Jv}{w}, \quad \braket{v}{w}_\sim=\iprod{\tilde Jv}{w},
  \quad v,w\in H\times H,
\end{equation}
with fundamental symmetries
\[J=\pmat{0&-iI\\iI&0},\qquad \tilde J=\pmat{0&I\\I&0}.\]
For $v=(x,y),w=(\tilde x,\tilde y)$ this yields
\begin{equation*}
  \braket{v}{w}=i\iprod{x}{\tilde y}-i\iprod{y}{\tilde x},
  \qquad\braket{v}{w}_\sim=\iprod{x}{\tilde y}+\iprod{y}{\tilde x}.
\end{equation*}
For the first inner product, we also consider its extension
to $v\in V_1=H_s\ads\times H_r$ and $w\in V_0=H_{-r}\times H_{-s}\ads$
which we denote again by $\braket{\cdot}{\cdot}$ and which is given by
\begin{equation}\label{eq:extJinner}
  \begin{aligned}
  &\braket{v}{w}=i\iprods{x}{\tilde y}{s}\ads-i\iprods{y}{\tilde x}{r},\\
  &\braket{w}{v}=i\iprodsr{\tilde x}{y}{r}-i\iprodsr{\tilde y}{x}{s}\ads=
  \overline{\braket{v}{w}}.
  \end{aligned}
\end{equation}
Note that the extended inner product is non-degenerate in the sense that
if $w\in V_0$ is such that $\braket{v}{w}=0$ for all $v\in V_1$,
then $w=0$.
Analogously $v\in V_1$ with $\braket{v}{w}=0$ for all $w\in V_0$ implies
$v=0$.

The Hamiltonian has the following properties with respect to
the inner products defined above:

\begin{lemma}\label{lem:Hamsym}
  \begin{gather*}
    \braket{T_0v}{w}=-\braket{v}{T_0w},\qquad v,w\in\mdef(T_0),\\
    \Real\braket{Tv}{v}_\sim\leq0, \qquad v\in\mdef(T).
  \end{gather*}
\end{lemma}

\begin{proof}
    Let $v,w\in\mdef(T_0)=H_{1-r}\ads\times H_{1-s}$ and
    $v=(x,y),w=(\tilde x,\tilde y)$. Then
    \[x,\tilde x\in H_{1-r}\ads\subset H_s\ads,\quad
    y,\tilde y\in H_{1-s}\subset H_r, \quad
    T_0v,T_0w\in V_0=H_{-r}\times H_{-s}\ads.\]
    We obtain
    \begin{align*}
      \braket{T_0v}{w}&=
      i\iprodsr{Ax-BB^*y}{\tilde y}{r}-i\iprodsr{-C^*Cx-A^*y}{\tilde x}{s}\ads\\
      &=i\iprodsr{Ax}{\tilde y}{r}-i\iprodsr{BB^*y}{\tilde y}{r}
      +i\iprodsr{C^*Cx}{\tilde x}{s}\ads+i\iprodsr{A^*y}{\tilde x}{s}\ads\\
      &=i\iprods{x}{A^*\tilde y}{s}\ads-i\iprods{y}{BB^*\tilde y}{r}
      +i\iprods{x}{C^*C\tilde x}{s}\ads+i\iprods{y}{A\tilde x}{r}\\
      &=i\iprods{x}{C^*C\tilde x+A^*\tilde y}{s}\ads
      -i\iprods{y}{-A\tilde x+BB^*\tilde y}{r}\\
      &=\braket{v}{-T_0w}.
    \end{align*}
    Let now $v=(x,y)\in\mdef(T)$. Then
    \begin{align*}
      \braket{Tv}{v}_\sim&=
      \iprod{Ax-BB^*y}{y}+\iprod{-C^*Cx-A^*y}{x}\\
      &=\iprodsr{Ax}{y}{r}-\iprodsr{BB^*y}{y}{r}
      -\iprodsr{C^*Cx}{x}{s}\ads-\iprodsr{A^*y}{x}{s}\ads\\
      &=\iprodsr{Ax}{y}{r}-\|B^*y\|_U^2-\|Cx\|_Y^2-\iprods{y}{Ax}{r}
    \end{align*}
    and hence
    \[\Real\braket{Tv}{v}_\sim=-\|B^*y\|_U^2-\|Cx\|_Y^2\leq0.\]
\end{proof}

\begin{coroll}\label{cor:specsym}
  \begin{enumerate}
  \item\label{cor:specsym:it1}
    If there exists $\lambda\in\C$ such that $\lambda,-\bar\lambda\in\varrho(T_0)$,
    then $T$ is $J$-skew-selfadjoint and $\sigma(T)$ is symmetric with respect
    to the imaginary axis.
  \item If both $T$ and $T_0$ have a compact resolvent, then
    $\sigma(T_0)$ is symmetric with respect to the imaginary axis.
  \end{enumerate}
\end{coroll}

\begin{proof}
  The previous lemma yields $\braket{Tv}{w}=-\braket{v}{Tw}$ for
  $v,w\in V$.
  Also recall that $T$ is densely defined since $\varrho(T_0)\neq\varnothing$.
  Lemma~\ref{lem:op2} implies $\varrho(T_0)\subset \varrho(T)$ and hence
  $\lambda,-\bar\lambda\in\varrho(T)$.
  By the theory of operators in Krein spaces, we conclude that $T$ is
  skew-selfadjoint with respect to the $J$-inner product, which
  in turn implies the symmetry of the spectrum.
  If now both resolvents are compact, then
  $\sigma(T)=\sigma_p(T)=\sigma_p(T_0)=\sigma(T_0)$ and the symmetry
  of the spectrum follows from part (a).
\end{proof}

\begin{remark}
  The symmetries of the Hamiltonian with respect to the two indefinite
  inner products on $H\times H$ have been used already in
  \cite{langer-ran-rotten,tretter-wyss,wyss-rinvsubham,wyss-unbctrlham}.
  The use of the Hamiltonian  $T_0$ on the extended space
  $V_0$ as well as the extended indefinite inner product
  is new here and is motivated by
  the better properties of $T_0$ compared to $T$.
\end{remark}

\section{Bisectorial Hamiltonians}\label{sec:bisecham}

Starting from this section we consider Hamiltonians whose
operator $A$ is quasi-sectorial, see Definition~\ref{def:sectorial}.
Recall from  Section~\ref{sec:ham} that
\[ V_1=H_s\ads\times H_r,\qquad
V_0=H_{-r}\times H_{-s}\ads\]
and
\begin{gather*}
  BB^*\in L(H_r, H_{-r}), \qquad
  C^*C\in L(H\ads_s, H\ads_{-s}),
\end{gather*}
We consider the following decomposition of $T_0$ on $V_0$:
\begin{equation}\label{eq:T0decomp}
  T_0=S_0+R, \qquad
  S_0=\pmat{A&0\\0&-A^*},\quad R=\pmat{0&-BB^*\\-C^*C&0}.
\end{equation}
Here $S_0$, like $T_0$, is an unbounded operator on $V_0$ with domain
$\mdef(S_0)=\mdef(T_0)=H_{1-r}\ads\times H_{1-s}$.
On the other hand, $R$ is a bounded operator $R\in L(V_1,V_0)$.

By Corollary~\ref{cor:scale-Anorms} the extensions of $A$ and $A^*$
to unbounded operators on $H_{-r}$ and $H_{-s}\ads$, respectively,
are quasi-sectorial and satisfy
\[\|(A-\lambda)^{-1}\|_{L(H_{-r})}\leq\frac{M}{|\lambda|},\qquad
\|(A^*-\lambda)^{-1}\|_{L(H_{-s}\ads)}\leq\frac{M}{|\lambda|}\]
for all
$\lambda\in\Sigma_{\frac\pi2+\theta}$, $|\lambda|\geq \rho$
where $\theta,M,\rho$ are the constants from \eqref{eq:quasisect}.
Consequently
\begin{equation}\label{eq:S0est}
  \|(S_0-\lambda)^{-1}\|_{L(V_0)}\leq\frac{M}{|\lambda|},\qquad \lambda\in\Omega_\theta,
  |\lambda|\geq\rho,
\end{equation}
with $\Omega_\theta$ the bisector from \eqref{eq:bisector}.

We derive a few estimates for the resolvents of $A$ and $A^*$
with respect to the scales of Hilbert spaces
$\{H_s\}$ and $\{H_s\ads\}$.
\begin{lemma}\label{lem:sectinterp}
  Let $A$ be quasi-sectorial and let $\theta,M,\rho>0$ be the corresponding
  constants from \eqref{eq:quasisect}.
  Then for all $\lambda\in\Sigma_{\frac\pi2+\theta}$ with $|\lambda|\geq \rho$
  the estimates
  \begin{gather*}
    \|(A-\lambda)^{-1}\|_{H\to H_1\ads}\leq M_1, \qquad
    \|(A-\lambda)^{-1}\|_{H_{-1}\to H}\leq M_1,\\
    \|(A^*-\lambda)^{-1}\|_{H\to H_1}\leq M_1, \qquad
    \|(A^*-\lambda)^{-1}\|_{H_{-1}\ads\to H}\leq M_1
  \end{gather*}
  hold where $M_1=M\left(\frac{1}{\rho}+1\right)+1$.
\end{lemma}
\begin{proof}
  For $x\in H$ we have
  \begin{align*}
    \|(A-\lambda)^{-1}x\|_1\ads
    &\leq\|(A-\lambda)^{-1}x\|+\|A(A-\lambda)^{-1}x\|\\
    &\leq\|(A-\lambda)^{-1}x\|+\|x\|+|\lambda|\|(A-\lambda)^{-1}x\|\\
    &\leq\left(\frac{M}{|\lambda|}+1+M\right)\|x\|
    \leq\left(\frac{M}{\rho}+1+M\right)\|x\|
  \end{align*}
  and hence $\|(A-\lambda)^{-1}\|_{H\to H_1\ads}\leq M_1$.
  Since the adjoint of $(A-\bar\lambda)^{-1}:H\to H_1\ads$
  with respect to the scale $\{H_s\ads\}$ is
  $(A^*-\lambda)^{-1}:H_{-1}\ads\to H$, see Section~\ref{sec:scale},
  we also get
  \[\|(A^*-\lambda)^{-1}\|_{H_{-1}\ads\to H}
  =\|(A-\bar\lambda)^{-1}\|_{H\to H_1\ads}\leq M_1.\]
  Note here that if $\lambda$ belongs to $\Sigma_{\frac\pi2+\theta}$
  then so does $\bar\lambda$.
  The other estimates follow by interchanging the roles of $A$ and $A^*$.
\end{proof}

\begin{coroll}\label{cor:sectinterp}
  Let $A$ be quasi-sectorial, $\theta,M,\rho$ as above.
  Let $r,s\geq0$ with $r+s\leq1$.
  Then for $\lambda\in\Sigma_{\frac\pi2+\theta}$, $|\lambda|\geq \rho$:
  \begin{gather*}
    \|(A-\lambda)^{-1}\|_{H_{-r}\to H_s\ads}\leq \frac{M_2}{|\lambda|^{1-r-s}}, \qquad
    \|(A^*-\lambda)^{-1}\|_{H_{-s}\ads\to H_r}\leq \frac{M_2}{|\lambda|^{1-r-s}}.
  \end{gather*}
  The constant $M_2$ depends on $M,\rho,r,s$ only.
\end{coroll}
\begin{proof}
  We apply interpolation to the results of Lemma~\ref{lem:sectinterp}.
  As a first step we get
  \begin{align*}
    \|(A-\lambda)^{-1}\|_{H\to H_{r+s}\ads}
    &\leq\|(A-\lambda)^{-1}\|_{H\to H_1\ads}^{r+s}\|(A-\lambda)^{-1}\|_{H\to H}^{1-r-s}\\
    &\leq M_1^{r+s}\left(\frac{M}{|\lambda|}\right)^{1-r-s}
    =\frac{M_2}{|\lambda|^{1-r-s}}
  \end{align*}
  with $M_2=M_1^{r+s}M^{1-r-s}$
  and similarly
  \[\|(A-\lambda)^{-1}\|_{H_{-r-s}\to H}\leq \frac{M_2}{|\lambda|^{1-r-s}}.\]
  From this we obtain
  with $\tau=\frac{r}{r+s}$
  \begin{align*}
    \|(A-\lambda)^{-1}\|_{H_{-r}\to H_s\ads}
    &\leq\|(A-\lambda)^{-1}\|_{H_{-r-s}\to H}^\tau\|(A-\lambda)^{-1}\|_{H\to H_{r+s}\ads}^{1-\tau}\\
    &\leq\frac{M_2}{|\lambda|^{1-r-s}}.
  \end{align*}
  The estimates for $\|(A^*-\lambda)^{-1}\|_{H_{-s}\ads\to H_r}$ are again analogous.
\end{proof}

\begin{lemma}\label{lem:T0pert}
  Let $A$ be quasi-sectorial, let $\theta,\rho$ be the constants from \eqref{eq:quasisect}.
  Suppose that $r+s<1$.
  Then there exists $\rho_1\geq \rho$ and $c_0,c_1>0$ such that
  $\Omega_\theta\setminus B_{\rho_1}(0)\subset\varrho(T_0)$
  and
  \begin{gather}
    \label{eq:T0est}
    \|(T_0-\lambda)^{-1}\|_{L(V_0)}\leq\frac{c_0}{|\lambda|},\\
    \label{eq:T0S0diffest}
    \|(T_0-\lambda)^{-1}-(S_0-\lambda)^{-1}\|_{L(V_0)}\leq\frac{c_1}{|\lambda|^{2-r-s}}
  \end{gather}
  for all $\lambda\in\Omega_\theta$, $|\lambda|\geq \rho_1$.
\end{lemma}
\begin{proof}
  This is a standard perturbation argument for $T_0=S_0+R$ on $V_0$:
  For $\lambda\in\varrho(S_0)$, the identity
  \[T_0-\lambda=\bigl(I-R(S_0-\lambda)^{-1}\bigr)(S_0-\lambda)\]
  holds.
  Corollary~\ref{cor:sectinterp} implies that
  \[\|(S_0-\lambda)^{-1}\|_{L(V_0,V_1)}\leq\frac{M_2}{|\lambda|^{1-r-s}}, \qquad
  \lambda\in\Omega_\theta,|\lambda|\geq\rho.\]
  Since $\|R(S_0-\lambda)^{-1}\|_{L(V_0)}\leq\|R\|\|(S_0-\lambda)^{-1}\|_{L(V_0,V_1)}$
  and $1-r-s>0$, it follows that
  there exists $\rho_1\geq\rho$ such that
  \[\|R(S_0-\lambda)^{-1}\|_{L(V_0)}\leq\frac12
  \qquad\text{for all }\lambda\in\Omega_\theta,\,|\lambda|\geq \rho_1.\]
  Hence $I-R(S_0-\lambda)^{-1}$ is an isomorphism on $V_0$
  and thus $\lambda\in\varrho(T_0)$ with
  \begin{equation}\label{eq:T0resolvpert}
    (T_0-\lambda)^{-1}=(S_0-\lambda)^{-1}\bigl(I-R(S_0-\lambda)^{-1}\bigr)^{-1}
  \end{equation}
  and
  \[\|(T_0-\lambda)^{-1}\|_{L(V_0)}\leq\|(S_0-\lambda)^{-1}\|_{L(V_0)}
  \bigl\|\bigl(I-R(S_0-\lambda)^{-1}\bigr)^{-1}\bigr\|_{L(V_0)}
  \leq\frac{2M}{|\lambda|}\]
  for $\lambda\in\Omega_\theta$, $|\lambda|\geq\rho_1$.
  Moreover
  \begin{equation}\label{eq:T0S0diff}
    (S_0-\lambda)^{-1}-(T_0-\lambda)^{-1}=(T_0-\lambda)^{-1}R(S_0-\lambda)^{-1}
  \end{equation}
  which implies
  \begin{align*}
    \|(S_0-\lambda)^{-1}-(T_0-\lambda)^{-1}\|_{L(V_0)}
    &\leq\|(T_0-\lambda)^{-1}\|_{L(V_0)}\|R\|\|(S_0-\lambda)^{-1}\|_{L(V_0,V_1)}\\
    &\leq\frac{2M\|R\|M_2}{|\lambda|^{2-r-s}}.
  \end{align*}
\end{proof}

\begin{lemma}\label{lem:SQpm}
  Let $A$ be quasi-sectorial and let $Q_{0\pm}\in L(V_0)$ be the projections
  \begin{equation}\label{eq:Qpm}
    Q_{0-}=\pmat{I&0\\0&0},\qquad Q_{0+}=\pmat{0&0\\0&I}.
  \end{equation}
  Consider the integration contours $\gamma_1(t)=it$,
  $t\in\,]-\infty,-\rho]\cup[\rho,\infty[\,$
  as well as
  $\gamma_{0+}(t)=\rho e^{it}$, $t\in[-\frac\pi2,\frac\pi2]$
  and
  $\gamma_{0-}(t)=\rho e^{-it}$, $t\in[\frac\pi2,\frac{3\pi}{2}]$
  where $\rho$ is the constant from \eqref{eq:quasisect} for $A$.
  Then
  \[Q_{0+}v-Q_{0-}v=\frac{1}{\pi i}\int_{\gamma_1}'(S_0-\lambda)^{-1}v\,d\lambda
  +Kv, \qquad v\in V_0,\]
  where the prime denotes the Cauchy principal value at infinity and
  $K\in L(V_0)$ is given by
  $K=\smallmat{K_1&0\\0&K_2}$ with
  \[K_1=\frac{1}{\pi i}\int_{\gamma_{0+}}(A-\lambda)^{-1}\,d\lambda, \quad
  K_2=\frac{1}{\pi i}\int_{\gamma_{0-}}(-A^*-\lambda)^{-1}\,d\lambda.\]
\end{lemma}
\begin{proof}
  We consider $A$ as an operator on $H_{-r}$.
  Since $A-\rho$ is sectorial and $0\in\varrho(A-\rho)$,
  \[\frac{1}{\pi i}\int_{-i\infty}^{i\infty\,\prime}(A-\rho-\lambda)^{-1}x
  \,d\lambda=-x, \qquad x\in H_{-r}, \]
  holds by \cite[Lemma~6.1]{langer-ran-rotten}.
  Using Cauchy's theorem in conjunction with
  the resolvent decay of $A$ to alter the integration contour, we obtain
  \begin{align*}
    -x
    &=\frac{1}{\pi i}\int_{\rho-i\infty}^{\rho+i\infty\,\prime}(A-\lambda)^{-1}x\,d\lambda\\
    &=\frac{1}{\pi i}\int_{\gamma_1}'(A-\lambda)^{-1}x\,d\lambda
    +\frac{1}{\pi i}\int_{\gamma_{0+}}(A-\lambda)^{-1}x\,d\lambda,
    \qquad x\in H_{-r}.
  \end{align*}
  Looking at $-A^*$, we get
  \[\frac{1}{\pi i}\int_{-i\infty}^{i\infty\,\prime}(-A^*+\rho-\lambda)^{-1}y
  \,d\lambda=y, \qquad y\in H_{-s}\ads, \]
  and hence
  \[y=\frac{1}{\pi i}\int_{\gamma_1}'(-A^*-\lambda)^{-1}y\,d\lambda
  +\frac{1}{\pi i}\int_{\gamma_{0-}}(-A^*-\lambda)^{-1}y\, d\lambda,
  \qquad y\in H_{-s}\ads.
  \]
  Combining both identities and noting that
  $Q_{0+}v-Q_{0-}v=(-x,y)$ for $v=(x,y)$, we obtain the claim.
\end{proof}

\begin{theo}\label{theo:T0dichot}
  Let $A$ be quasi-sectorial and let $r+s<1$.
  If $\sigma(A)\cap i\R=\varnothing$ or if $A$ has a compact resolvent and
  \begin{equation}\label{eq:T0dichot-cond}
      \ker(A-it)\cap\ker C=\ker(A^*+it)\cap\ker B^*=\{0\}
      \quad\text{for all}\quad t\in\R,
  \end{equation}
  then the Hamiltonian $T_0$ is bisectorial and strictly dichotomous.
\end{theo}
\begin{proof}
  We first show that $i\R\subset\varrho(T_0)$.
  If $\sigma(A)\cap i\R=\varnothing$, then
  Lemma~\ref{lem:apprgap} implies
  $\sigma_\mathrm{app}(T_0)\cap i\R=\varnothing$.
  Since $\partial\sigma(T_0)\subset\sigma_\mathrm{app}(T_0)$ and
  since $i\R\cap\varrho(T_0)\neq\varnothing$ by Lemma~\ref{lem:T0pert}
  it follows that $i\R\subset\varrho(T_0)$.
  Suppose on the other hand that $A$ has a compact resolvent
  and that \eqref{eq:T0dichot-cond} holds.
  By Lemma~\ref{lem:ham-cptresolv}
  $T_0$ has a compact resolvent too
  and therefore $\sigma(T_0)=\sigma_p(T_0)$.
  Lemma~\ref{lem:pointgap} then implies $\sigma(T_0)\cap i\R=\varnothing$.

  From $i\R\subset\varrho(T_0)$ and the estimate \eqref{eq:T0est}
  we obtain that $T_0$ is bisectorial.
  In particular Theorem~\ref{theo:sectdichot} can be applied to $T_0$
  and yields corresponding closed projections on $V_0$, which we
  denote by $P_{0\pm}$.
  By Lemma~\ref{lem:SQpm} the mapping
  \[v\mapsto\frac{1}{\pi i}\int_{\gamma_1}'(S_0-\lambda)^{-1}v\,d\lambda,
  \qquad v\in V_0,\]
  defines a bounded operator in $L(V_0)$.
  In view of \eqref{eq:T0S0diffest} the integral
  \[\int_{\gamma_1}(T_0-\lambda)^{-1}-(S_0-\lambda)^{-1}\,d\lambda\]
  converges in $L(V_0)$.
  Consequently
  $v\mapsto\frac{1}{\pi i}\int_{\gamma_1}'(T_0-\lambda)^{-1}v\,d\lambda$
  and hence also
  \[v\mapsto\frac{1}{\pi i}\int_{-i\infty}^{i\infty\,\prime}(T_0-\lambda)^{-1}v\,d\lambda,
  \qquad v\in V_0,\]
  defines a bounded operator in $L(V_0)$.
  By \eqref{eq:dichotint2} this last operator coincides with $P_{0+}-P_{0-}$ on $\mdef(T_0)$.
  Since $P_{0+}-P_{0-}$ is  closed and $\mdef(T_0)$ is dense in $V_0$,
  we conclude that $\mdef(P_{0\pm})=V_0$ and hence $P_{0\pm}\in L(V_0)$ by the closed
  graph theorem.
  Therefore $T_0$ is strictly dichotomous.  
\end{proof}

\begin{remark}
  Combining the results from Lemma~\ref{lem:T0pert} with the dichotomy of $T_0$
  from Theorem~\ref{theo:T0dichot}
  we find that in fact
  \[\bigl(\Omega_\theta\setminus B_{\rho_1}(0)\bigr)\cup\bigset{\lambda\in\C}{|\lambda|\leq h}
  \subset\varrho(T_0)\]
  where $\rho_1\geq\rho$, $h>0$, and $\theta,\rho$ are the constants from
  \eqref{eq:quasisect} corresponding to the quasi-sectoriality of $A$.
  Also note that the last proof shows that $T_0$ is bisectorial and strictly
  dichotomous whenever $r+s<1$ and $i\R\subset\varrho(T_0)$.
\end{remark}

We close this section by investigating the dichotomy properties of
the Hamiltonian on $V=H\times H$, i.e., of the operator $T$.
Let
\[S=\pmat{A&0\\0&-A^*}\]
with domain $\mdef(S)=H_1\ads\times H_1$,
considered as an unbounded operator on $V$,
i.e., $S$ is the part of $S_0$ in $V$.
Note that a decomposition similar to \eqref{eq:T0decomp} does not hold for the
operators $T$ and $S$ since
$R$ maps out of $V$ into the larger space $V_0$.
In particular we have $\mdef(T)\neq\mdef(S)$ in general.

\begin{lemma}\label{lem:Tpert}
  Let $A$ be quasi-sectorial with constants $\theta,\rho$ as in \eqref{eq:quasisect}.
  Let $r+s<1$.
  Then there exist $\rho_1\geq \rho$ and $c_0,c_1>0$ such that
  $\Omega_\theta\setminus B_{\rho_1}(0)\subset\varrho(T)$ and
  \begin{gather}
    \label{eq:Test}
    \|(T-\lambda)^{-1}\|_{L(V)}\leq\frac{c_0}{|\lambda|^\beta},\\
    \label{eq:TSdiffest}
    \|(T-\lambda)^{-1}-(S-\lambda)^{-1}\|_{L(V)}\leq\frac{c_1}{|\lambda|^{2(1-\max\{r,s\})}},
  \end{gather}
  for all $\lambda\in\Omega_\theta$, $|\lambda|\geq\rho_1$
  where
  \[\beta=\begin{cases}1,&\max\{r,s\}\leq\frac12,\\
  2(1-\max\{r,s\}),&\max\{r,s\}>\frac12.
  \end{cases}\]
\end{lemma}
\begin{proof}
  By Corollary~\ref{cor:sectinterp} there exist $M_2,M_2'>0$ with
  \[\|(A-\lambda)^{-1}\|_{L(H_{-r},H)}\leq\frac{M_2}{|\lambda|^{1-r}},
  \qquad
  \|(-A^*-\lambda)^{-1}\|_{L(H_{-s}\ads,H)}\leq
  \frac{M_2'}{|\lambda|^{1-s}}\]
  for all $\lambda\in\Omega_\theta$, $|\lambda|\geq\rho$.
  Since $\rho>0$ we can thus find $c>0$ such that
  \[\|(S_0-\lambda)^{-1}\|_{L(V_0,V)}\leq\frac{c}{|\lambda|^{1-\max\{r,s\}}}
  \qquad\text{for }\lambda\in\Omega_\theta,\,|\lambda|\geq\rho.\]
  Similarly there exists $c'>0$ with
  \[\|(S-\lambda)^{-1}\|_{L(V,V_1)}\leq\frac{c'}{|\lambda|^{1-\max\{r,s\}}}
  \qquad\text{for }\lambda\in\Omega_\theta,\,|\lambda|\geq\rho.\]
  Let now $\rho_1\geq\rho$ be chosen as in Lemma~\ref{lem:T0pert}
  and let $\lambda\in\Omega_\theta$, $|\lambda|\geq\rho_1$.
  Then $\lambda\in\varrho(T_0)$ and we obtain from \eqref{eq:T0resolvpert}
  that
  \begin{equation}\label{eq:T0resolvest-V}
  \begin{aligned}
    \|(T_0-\lambda)^{-1}\|_{L(V_0,V)}
    &\leq\|(S_0-\lambda)^{-1}\|_{L(V_0,V)}
    \bigl\|\bigl(I-R(S_0-\lambda)^{-1}\bigr)^{-1}\bigr\|_{L(V_0)} \\
    &\leq\frac{2c}{|\lambda|^{1-\max\{r,s\}}}
  \end{aligned}
  \end{equation}
  and consequently
  \begin{align}\label{eq:T0RSest}
    \|(T_0-\lambda)^{-1}R(S-\lambda)^{-1}\|_{L(V)}
    &\leq\|(T_0-\lambda)^{-1}\|_{L(V_0,V)}\|R\|\|(S-\lambda)^{-1}\|_{L(V,V_1)}
    \notag\\
    &\leq\frac{2cc'\|R\|}{|\lambda|^{2(1-\max\{r,s\})}}.
  \end{align}
  Lemma~\ref{lem:op2} implies that
  $\lambda\in\varrho(T)$ and $(T-\lambda)^{-1}=(T_0-\lambda)^{-1}|_V$.
  Restricting \eqref{eq:T0S0diff} to the space $V$, we get
  \begin{equation}\label{eq:TSdiff}
    (S-\lambda)^{-1}-(T-\lambda)^{-1}=(T_0-\lambda)^{-1}R(S-\lambda)^{-1}.
  \end{equation}
  Combining this with \eqref{eq:T0RSest} and 
  $\|(S-\lambda)^{-1}\|_{L(V)}\leq M/|\lambda|$,
  we obtain the desired estimates.
\end{proof}

\begin{remark}\label{rem:SQpm}
  The statement of Lemma~\ref{lem:SQpm} remains true
  if all involved operators are restricted to $V$.
  This means that
  $V_0$, $S_0$ and $Q_{0\pm}$ are replaced by
  $V$, $S$ and $Q_\pm$, respectively, where
  $Q_\pm$ are the restrictions of $Q_{0\pm}$ to $V$.
  The proof remains unchanged
  except for an adaption of  the spaces.
\end{remark}

\begin{theo}\label{theo:Tdichot}
  Let $A$ be quasi-sectorial and let $r+s<1$.
  If $\sigma(A)\cap i\R=\varnothing$ or if $A$ has a compact resolvent and
  \begin{equation*}
      \ker(A-it)\cap\ker C=\ker(A^*+it)\cap\ker B^*=\{0\}
      \quad\text{for all}\quad t\in\R,
  \end{equation*}
  then  $T$ is almost bisectorial; in particular there exist closed,
  $T$- and $(T-\lambda)^{-1}$-invariant subspaces $V_\pm\subset V$
  such that $\sigma(T|_{V_\pm})\subset\C_\pm$.
  If in addition $\max\{r,s\}<\frac12$, then $T$ is even
  bisectorial and strictly dichotomous.
\end{theo}
\begin{proof}
  From Theorem~\ref{theo:T0dichot} we know that $i\R\subset\varrho(T_0)$.
  Hence also $i\R\subset\varrho(T)$ by Lemma~\ref{lem:op2}.
  From \eqref{eq:Test} in Lemma~\ref{lem:Tpert}
  we thus conclude that $T$ is almost bisectorial
  with $0<\beta<1$ if $\max\{r,s\}>\frac12$ and bisectorial
  if $\max\{r,s\}\leq\frac12$. Note that bisectoriality implies
  almost bisectoriality
  here since $0\in\varrho(T)$.
  The existence of $V_\pm$ follows by Theorem~\ref{theo:sectdichot}.
  If now $\max\{r,s\}<\frac12$ then \eqref{eq:TSdiffest} yields
  \[\|(T-\lambda)^{-1}-(S-\lambda)^{-1}\|\leq\frac{c_1}{|\lambda|^{1+\eps}},
  \qquad\lambda\in\Omega_\theta,\,|\lambda|\geq\rho_1,\]
  with some $\eps>0$.
  In view of Remark~\ref{rem:SQpm} we can then derive in the same way
  as in the proof of Theorem~\ref{theo:T0dichot}
  that $T$ is dichotomous.
\end{proof}

\section{Graph and angular subspaces}
\label{sec:angular}

In this section we consider a Hamiltonian with quasi-sectorial $A$,
$r+s<1$, and $i\R\subset\varrho(T_0)$. From the last section
we know that then $T_0$ is bisectorial and strictly dichotomous
and $T$ is almost bisectorial.
We denote by $V_{0\pm}$ and $V_\pm$ the corresponding invariant subspaces
of $T_0$ and $T$, respectively,
and by $P_{0\pm}$ and $P_\pm$ the associated projections;
see Theorem~\ref{theo:sectdichot}.
In particular $P_{0\pm}\in L(V_0)$ while $P_\pm$ are closed operators on $V$.
The projections $P_{0\pm}$ are given by $P_{0\pm}=T_0L_{0\pm}$
where $L_{0\pm}\in L(V_0)$,
\begin{equation}\label{eq:L0pm}
  L_{0\pm}=\frac{\pm 1}{2\pi i}\int_{\pm h-i\infty}^{\pm h+i\infty}
  \frac{1}{\lambda}(T_0-\lambda)^{-1}\,d\lambda.
\end{equation}

Recall from \eqref{eq:extJinner} 
the extended indefinite inner product
$\braket{\cdot}{\cdot}$ defined on $V_1\times V_0$ as well as $V_0\times V_1$.

\begin{lemma}\label{lem:L0pm}
  The operators $L_{0\pm}$ satisfy
  $L_{0\pm}\in L(V_0,V_1)$ and 
  \[\braket{L_{0+}v}{w}=-\braket{v}{L_{0-}w}
  \qquad\text{for all }v,w\in V_0.\]
\end{lemma}
\begin{proof}
  In the proof of Lemma~\ref{lem:T0pert} we have seen that there exists
  $\rho_1>0$ such that
  \[(T_0-\lambda)^{-1}=(S_0-\lambda)^{-1}
  \bigl(I-R(S_0-\lambda)^{-1}\bigr)^{-1}\]
  for $\lambda\in\Omega_\theta$, $|\lambda|>\rho_1$,
  and the estimates
  \[\|(S_0-\lambda)^{-1}\|_{L(V_0,V_1)}\leq\frac{M_2}{|\lambda|^{1-r-s}},
  \qquad
  \|R(S_0-\lambda)^{-1}\|_{L(V_0)}\leq\frac12\]
  hold.
  It follows that
  \begin{equation}\label{eq:T0est-V0V1}
    \|(T_0-\lambda)^{-1}\|_{L(V_0,V_1)}\leq\frac{2M_2}{|\lambda|^{1-r-s}}.
  \end{equation}
  Since $1-r-s>0$ this implies that the integral in \eqref{eq:L0pm}
  converges in $L(V_0,V_1)$; in particular $L_{0\pm}\in L(V_0,V_1)$.
  For $v,w\in V_0$ we can now derive, using Lemma~\ref{lem:Hamsym},
  \begin{align*}
    \braket{L_{0+}v}{w}
    &=\Braket{\frac{1}{2\pi i}\int_{h-i\infty}^{h+i\infty}\frac{1}{\lambda}
      (T_0-\lambda)^{-1}v\,d\lambda}{w}\\
    &=\frac{1}{2\pi}\int_{-\infty}^\infty
    \Braket{\frac{1}{h+it}(T_0-h-it)^{-1}v}{w}\,dt\\
    &=\frac{1}{2\pi}\int_{-\infty}^\infty
    \Braket{v}{\frac{1}{h-it}(-T_0-h+it)^{-1}w}\,dt\\
    &=\frac{1}{2\pi}\int_{-\infty}^\infty
    \Braket{v}{\frac{1}{-h+it}(T_0+h-it)^{-1}w}\,dt\\
    &=\Braket{v}{\frac{1}{2\pi i}\int_{-h-i\infty}^{-h+i\infty}
      \frac{1}{\lambda}(T_0-\lambda)^{-1}w\,d\lambda}
    =-\braket{v}{L_{0-}w}.
  \end{align*}
\end{proof}

\begin{coroll}\label{cor:V0pmneutral}
  \[\braket{v}{w}=0 \qquad\text{for all}\quad v\in V_{0\pm},\, w\in\range(L_{0\pm}).\]
\end{coroll}
\begin{proof}
  This is immediate since $V_{0\pm}=\ker L_{0\mp}$.
\end{proof}

We can now establish conditions for the subspaces $V_{0\pm}$ to be
graphs of operators.
We say that a subspace $U\subset V_0=H_{-r}\times H_{-s}\ads$
is the graph of a (possibly unbounded) operator
$X:\mdef(X)\subset H_{-r}\to H_{-s}\ads$ if
\[U=\set{\pmat{x\\Xx}}{x\in\mdef(X)}=\range\pmat{I\\X}.\]
We also consider the inverse situation where $U\subset H_{-r}\times H_{-s}\ads$
is the graph of an
operator $Y:\mdef(Y)\subset H_{-s}\ads\to H_{-r}$, i.e.,
\[U=\set{\pmat{Yy\\y}}{y\in\mdef(Y)}=\range\pmat{Y\\I}.\]

\begin{prop}\label{prop:graph-acao}
  If
  \begin{equation}\label{eq:ac}
    \bigcap_{\lambda\in i\R\cap\varrho(A^*)}\ker B^*(A^*-\lambda)^{-1}=\{0\}
    \qquad\text{on}\quad H_{-s}\ads,
  \end{equation}
  then $V_{0\pm}=\range\pmat{I\\X_{0\pm}}$ with closed operators
  $X_{0\pm}:\mdef(X_{0\pm})\subset H_{-r}\to H_{-s}\ads$.
  If
  \begin{equation}\label{eq:ao}
    \bigcap_{\lambda\in i\R\cap\varrho(A)}\ker C(A-\lambda)^{-1}=\{0\}
    \qquad\text{on}\quad H_{-r},
  \end{equation}
  then $V_{0\pm}=\range\pmat{Y_{0\pm}\\I}$ with closed operators
  $Y_{0\pm}:\mdef(Y_{0\pm})\subset H_{-s}\ads\to H_{-r}$.
  If both \eqref{eq:ac} and \eqref{eq:ao} hold then $X_{0\pm}$ are
  injective and $X_{0\pm}^{-1}=Y_{0\pm}$.
\end{prop}
\begin{proof}
  For the first assertion, since $V_{0\pm}$ are closed linear subspaces
  of $V_0$, it suffices to show that $(0,w)\in V_{0\pm}$ implies $w=0$.
  Let $(0,w)\in V_{0\pm}$ and $t\in\R$ such that $-it\in\varrho(A^*)$.
  Set
  \[\pmat{x\\y}=(T_0-it)^{-1}\pmat{0\\w}.\]
  Then $(x,y)\in\mdef(T_0)\cap V_{0\pm}$ by the invariance of $V_{0\pm}$.
  By Lemma~\ref{lem:Lpm-range}  it follows that
  \((x,y)\in\range(L_{0\pm})\).
  Using Corollary~\ref{cor:V0pmneutral}, we get
  \[0=\Braket{\pmat{x\\y}}{\pmat{0\\w}}=i\iprods{x}{w}{s}\ads.\]
  From
  \[\pmat{0\\w}=(T_0-it)\pmat{x\\y}=\pmat{(A-it)x-BB^*y\\-C^*Cx-(A^*+it)y}\]
  we thus obtain
  \begin{align*}
    0&=\iprods{x}{w}{s}\ads
    =-\iprods{x}{C^*Cx}{s}\ads-\iprods{x}{(A^*+it)y}{s}\ads\\
    &=-\|Cx\|^2-\iprodsr{(A-it)x}{y}{r}
    =-\|Cx\|^2-\iprodsr{BB^*y}{y}{r}\\
    &=-\|Cx\|^2-\|B^*y\|^2
  \end{align*}
  and therefore $Cx=B^*y=0$.
  This implies $w=-(A^*+it)y$ and hence $-B^*y=B^*(A^*+it)^{-1}w=0$.
  Since $t\in\R$ with $-it\in\varrho(A^*)$ was arbitrary, \eqref{eq:ac}
  implies that $w=0$.
  For the second assertion, we show in an analogous way
  that $(w,0)\in V_{0\pm}$ implies $w=0$ provided that \eqref{eq:ao} holds.
  The final statement is then clear.
\end{proof}

\begin{prop}\label{prop:graph-stable}
  Suppose that $A$ is sectorial with $0\in\varrho(A)$. Then
  \[V_{0-}=\range\pmat{I\\X_{0-}},\qquad
  V_{0+}=\range\pmat{Y_{0+}\\I}\]
  with closed operators $X_{0-}:\mdef(X_{0-})\subset H_{-r}\to H_{-s}\ads$
  and $Y_{0+}:\mdef(Y_{0+})\subset H_{-s}\ads\to H_{-r}$.
\end{prop}
\begin{proof}
  Let $(0,w)\in V_{0-}$ and $t\in\R$.
  Proceeding as in the previous proof, we set
  \[\pmat{x\\y}=(T_0-it)^{-1}\pmat{0\\w}\]
  and obtain $Cx=B^*y=0$ and hence
  $(A-it)x=0$ and $w=-(A^*+it)y$.
  Since $i\R\subset\varrho(A)$ it follows that
  \[(T_0-it)^{-1}\pmat{0\\w}=\pmat{x\\y}=\pmat{0\\(-A^*-it)^{-1}w}.\]
  We consider now the two functions
  \[\varphi(\lambda)=(T_0-\lambda)^{-1}\pmat{0\\w},\qquad
  \psi(\lambda)=\pmat{0\\(-A^*-\lambda)^{-1}w}.\]
  $\varphi$ is analytic on a strip $\set{\lambda\in\C}{|\Real\lambda|<\eps}$
  while
  $\psi$ is analytic on a half-plane
  $\set{\lambda\in\C}{\Real\lambda<\eps}$
  where $\eps>0$ is sufficiently small.
  The above derivation shows that $\varphi$ and $\psi$ coincide on $i\R$.
  Hence they coincide for $|\Real\lambda|<\eps$ by the identity theorem.
  Moreover $\psi$ is bounded on $\overline{\C_-}$
  since $A$ is sectorial with $0\in\varrho(A)$.
  On the other hand $\varphi$ extends to a bounded analytic function on $\overline{\C_+}$
  since $(0,w)\in V_{0-}$, see Theorem~\ref{theo:sectdichot}.
  Therefore $\varphi$ extends to a bounded entire function
  and is thus constant by
  Liouville's theorem. This implies $w=0$.

  Similarly for $(w,0)\in V_{0+}$, $t\in\R$ and
  \[\pmat{x\\y}=(T_0-it)^{-1}\pmat{w\\0}\]
  we derive
  $Cx=B^*y=0$, $w=(A-it)x$ and $(A^*+it)y=0$; hence
  \[(T_0-it)^{-1}\pmat{w\\0}=\pmat{(A-it)^{-1}w\\0}.\]
  In this case the analytic functions
  \[\varphi(\lambda)=(T_0-\lambda)^{-1}\pmat{w\\0},\qquad
  \psi(\lambda)=\pmat{(A-\lambda)^{-1}w\\0}\]
  coincide on $i\R$, $\varphi$ is bounded on $\overline{\C_-}$ since
  $(w,0)\in V_{0+}$, and $\psi$ is bounded on $\overline{\C_+}$.
  Therefore $\varphi$ is again constant and hence $w=0$.
\end{proof}

We turn to the question of the boundedness of the operators $X_{0\pm}$,
$Y_{0\pm}$.
To this end we use the concept of {angular subspaces},
see \cite[\S5.1]{bart-gohberg-kaashoek79}, \cite[Lemma~7.1]{wyss-unbctrlham}.
Consider again the projections from Lemma~\ref{lem:SQpm},
\[Q_{0-}=\pmat{I&0\\0&0},\qquad Q_{0+}=\pmat{0&0\\0&I},\]
acting on $V_0=H_{-r}\times H_{-s}\ads$.

\begin{lemma}
  Let $U$ be a closed subspace of $V_0$. Then:
  \begin{enumerate}
  \item
    $U=\range\smallmat{I\\X}$ with a closed operator
    $X:\mdef(X)\subset H_{-r}\to H_{-s}\ads$
    if and only if
    \[U\cap\ker Q_{0-}=\{0\}.\]
    $U=\range\smallmat{I\\X}$ with a {bounded} operator $X\in L(H_{-r},H_{-s}\ads)$
    if and only if
    \begin{equation}\label{eq:angularX}
      V_0=U\oplus\ker Q_{0-}.
    \end{equation}
  \item
    $U=\range\smallmat{Y\\I}$ with a closed operator
    $Y:\mdef(Y)\subset H_{-s}\ads\to H_{-r}$
    if and only if
    \[U\cap\ker Q_{0+}=\{0\}.\]
    $U=\range\smallmat{Y\\I}$ with $Y\in L(H_{-s}\ads,H_{-r})$
    if and only if
    \begin{equation}\label{eq:angularY}
      V_0=U\oplus\ker Q_{0+}.
    \end{equation}
  \end{enumerate}
\end{lemma}
\begin{proof}
  Observe that $\ker Q_{0-}=\{0\}\times H_{-s}\ads$. Since $U$ is the graph of some
  closed operator $X:\mdef(X)\subset H_{-r}\to H_{-s}\ads$ if and only if
  $(0,y)\in U$ implies $y=0$, the first assertion of (a) follows.
  By \cite[Proposition~5.1]{bart-gohberg-kaashoek79},
  \eqref{eq:angularX} holds if and only if
  $U=\set{Xx+x}{x\in\range(Q_{0-})}$ with $X\in L(\range(Q_{0-}),\ker Q_{0-})$.
  Identifying $\range(Q_{0-})\cong H_{-r}$ and $\ker Q_{0-}\cong H_{-s}\ads$,
  we obtain the second assertion of (a).
  The proof of (b) is analogous; here
  $\range(Q_{0+})\cong H_{-s}\ads$, $\ker Q_{0+}\cong H_{-r}$.
\end{proof}

If \eqref{eq:angularX} holds
then  $U$ is called \emph{angular} with respect to
$Q_{0-}$ and $X$ is
the \emph{angular operator} for $U$.
Similarly in case of  \eqref{eq:angularY},
$U$ is called angular with respect to $Q_{0+}$ and
angular operator $Y$.

The next lemma is the key step in proving that $V_{0\pm}$ are
angular subspaces.
The idea for its proof
goes back to \cite[Theorem~2.3]{bubak-mee-ran}
where instead of $F_1$ and $F_2$ the operator $Q_{0-}P+Q_{0+}\tilde P$
was used,
see also \cite[\S6.4]{bart-gohberg-kaashoek79}.
\begin{lemma}\label{lem:angular}
  Suppose $V_0=U\oplus \tilde U$ with closed subspaces $U,\tilde U\subset V_0$.
  Let $P,\tilde P\in L(V_0)$ be the associated complementary projections,
  $U=\range(P)$, $\tilde U=\range(\tilde P)$, $I=P+\tilde P$.
  Let $F_1=I-Q_{0-}+P$ and $F_2=I-P+Q_{0-}$.
  \begin{enumerate}
  \item If
    \begin{equation}\label{eq:V0pm-angular}
      U=\range\pmat{I\\X}, \qquad \tilde U=\range\pmat{Y\\I}
    \end{equation}
    with some $X:\mdef(X)\subset H_{-r}\to H_{-s}\ads$
    and $Y:\mdef(Y)\subset H_{-s}\ads\to H_{-r}$,
    then $F_1$ and $F_2$ are injective.
  \item If $F_1$ and $F_2$ are bijective, then
    \eqref{eq:V0pm-angular} holds with bounded operators
    $X\in L(H_{-r},H_{-s}\ads)$, $Y\in L(H_{-s}\ads,H_{-r})$.
  \end{enumerate}
\end{lemma}
\begin{proof}
  \begin{enumerate}
  \item By the previous lemma, identity \eqref{eq:V0pm-angular} implies that
    $U\cap\ker Q_{0-}=\tilde U\cap\ker Q_{0+}=\{0\}$.
    Let $F_1v=0$. Then $(I-Q_{0-})v=-Pv\in U\cap\ker Q_{0-}$,
    which implies $(I-Q_{0-})v=Pv=0$. It follows that
    $v\in\range(Q_{0-})\cap\ker P=\ker Q_{0+}\cap \tilde U$
    and hence $v=0$. The injectivity of $F_2$ is analogous.
  \item Let $v\in U\cap\ker Q_{0-}$. Then $(I-P)v=Q_{0-}v=0$,
    which yields $F_2v=0$ and thus $v=0$.
    On the other hand we can write $w\in V_0$ as $w=F_1v=(I-Q_{0-})v+Pv$
    and so $w\in U+\ker Q_{0-}$.
    This shows that $V_0=U\oplus \ker Q_{0-}$, i.e., $U$ is angular
    with respect to $Q_{0-}$.
    Since $F_1=I-\tilde P+Q_{0+}$ and $F_2=I-Q_{0+}+\tilde P$, we get by symmetry that
    $\tilde U$ is angular to $Q_{0+}$.
    The assertion follows by the previous lemma.
  \end{enumerate}
\end{proof}

\begin{coroll}\label{cor:angular}
  Suppose that $P_{0-}-Q_{0-}$ is compact.
  If
  \[V_{0-}=\range\pmat{I\\X_{0-}}, \qquad V_{0+}=\range\pmat{Y_{0+}\\I},\]
  with some operators $X_{0-},Y_{0+}$, then these operators are in fact bounded,
  $X_{0-}\in L(H_{-r},H_{-s}\ads)$, $Y_{0+}\in L(H_{-s}\ads, H_{-r})$.
\end{coroll}
\begin{proof}
  We use the previous lemma with $U=V_{0-}$, $\tilde U=V_{0+}$,
  $P=P_{0-}$, $\tilde P=P_{0+}$.
  Then $F_1=I+(P_{0-}-Q_{0-})$ and $F_2=I-(P_{0-}-Q_{0-})$,
  and the assertion follows from Fredholm's alternative.
\end{proof}

\begin{theo}\label{theo:angular-acao}
  Suppose that $A$ has a compact resolvent. If
  \begin{equation}\label{eq:ac-theo}
    \bigcap_{\lambda\in i\R\cap\varrho(A^*)}\ker B^*(A^*-\lambda)^{-1}=\{0\}
    \qquad\text{on}\quad H_{-s}\ads,
  \end{equation}
  and
  \begin{equation}\label{eq:ao-theo}
    \bigcap_{\lambda\in i\R\cap\varrho(A)}\ker C(A-\lambda)^{-1}=\{0\}
    \qquad\text{on}\quad H_{-r},
  \end{equation}
  then $V_{0\pm}=\range\pmat{I\\X_{0\pm}}$ where
  the operators $X_{0-}$ and $X_{0+}$ are injective,
  $X_{0-}\in L(H_{-r}, H_{-s}\ads)$
  and $X_{0+}^{-1}\in L(H_{-s}\ads,H_{-r})$.
\end{theo}
\begin{proof}
  If $A$ has a compact resolvent, then the same is true for
  $S_0$ and $T_0$, compare Lemma~\ref{lem:ham-cptresolv}.
  From Theorem~\ref{theo:sectdichot} and Lemma~\ref{lem:SQpm}  we
  know that
  \begin{equation*}
    P_{0+}v-P_{0-}v=\frac{1}{\pi i}\int_{-i\infty}^{i\infty\,\prime}
    (T_0-\lambda)^{-1}v\,d\lambda,
    \qquad v\in\mdef(T_0),
  \end{equation*}
  \[Q_{0+}v-Q_{0-}v=\frac{1}{\pi i}\int_{\gamma_1}'(S_0-\lambda)^{-1}v\,d\lambda
  +Kv, \qquad v\in V_0,\]
  where $K\in L(V_0)$.
  Since
  \[Q_{0+}-Q_{0-}-(P_{0+}-P_{0-})=I-2Q_{0-}-(I-2P_{0-})
  =2(P_{0-}-Q_{0-})\]
  we find
  \begin{multline*}
    2(P_{0-}-Q_{0-})v=
    \frac{1}{\pi i}\int_{\gamma_1}
    (S_0-\lambda)^{-1}-(T_0-\lambda)^{-1}\,d\lambda \,v\\
    -\frac{1}{\pi i}\int_{-i\rho}^{i\rho}(T_0-\lambda)^{-1}\,d\lambda \,v
    +Kv
  \end{multline*}
  for $v\in\mdef(T_0)$. Note here that because of \eqref{eq:T0S0diffest}
  the first integral converges in the operator norm topology of $L(V_0)$.
  In particular, both integrals on the right-hand side define bounded operators
  in $L(V_0)$ and hence the above identity holds for all $v\in V_0$.
  Since $(T_0-\lambda)^{-1}$ and $(S_0-\lambda)^{-1}$ are compact,
  both integrals yield in fact compact operators.
  The expression for $K$ in Lemma~\ref{lem:SQpm} implies that $K$ is compact too.
  Consequently $P_{0-}-Q_{0-}$ is compact.
  The assertion is now a consequence of Proposition~\ref{prop:graph-acao}
  and Corollary~\ref{cor:angular}.
\end{proof}

\begin{theo}\label{theo:angular-stable}
  Suppose that $A$ has a compact resolvent, is sectorial and $0\in\varrho(A)$.
  Then
  \[V_{0-}=\range\pmat{I\\X_{0-}}, \qquad V_{0+}=\range\pmat{Y_{0+}\\I}\]
  with $X_{0-}\in L(H_{-r},H_{-s}\ads)$, $Y_{0+}\in L(H_{-s}\ads, H_{-r})$.
\end{theo}
\begin{proof}
  As in the previous theorem we obtain that $P_{0-}-Q_{0-}$ is compact.
  Hence Proposition~\ref{prop:graph-stable} and Corollary~\ref{cor:angular}
  complete the proof.
\end{proof}

Next we investigate the graph  properties of the invariant subspaces $V_\pm$
of $T$. We know that $V_\pm=\range(P_\pm)$ where $P_\pm$ are the closed
projections on $V$ given by $P_\pm=TL_\pm$ with $L_\pm\in L(V)$,
\[L_\pm=\frac{\pm 1}{2\pi i}\int_{\pm h-i\infty}^{\pm h+i\infty}
\frac{1}{\lambda}(T-\lambda)^{-1}\,d\lambda.\]
In particular $L_\pm$ are the restrictions of $L_{0\pm}$ to $V$.
Since $V_\pm=\ker L_\mp$ and
$\ker L_\mp=\ker L_{0\mp}\cap V$ it follows that
\begin{equation}\label{eq:V0pm-Vpm-rel}
  V_\pm= V_{0\pm}\cap V.
\end{equation}
This implies that graph subspace structures of $V_{0\pm}$ are inherited by
the spaces $V_\pm$:

\begin{lemma}\label{lem:graphVpm}
  If
  \[V_{0+}=\range\pmat{I\\X_{0+}}\]
  with a closed operator
  $X_{0+}:\mdef(X_{0+})\subset H_{-r}\to H_{-s}\ads$,
  then also
  \[V_+=\range\pmat{I\\X_+}\]
  where
  $X_+:\mdef(X_+)\subset H\to H$ is closed and is the part of
  $X_{0+}$ in $H$, i.e.\ 
  \(\mdef(X_+)=\set{x\in\mdef(X_{0+})\cap H}{X_{0+}x\in H}\).
  Similarly, if
  \[V_{0+}=\range\pmat{Y_{0+}\\I}\]
  with a closed operator
  $Y_{0+}:\mdef(Y_{0+})\subset H_{-s}\ads\to H_{-r}$,
  then
  \[V_+=\range\pmat{Y_+\\I}\]
  where
  $Y_+:\mdef(Y_+)\subset H\to H$ is closed and is the part of
  $Y_{0+}$ in $H$.
  The corresponding statements hold for $V_{0-}$ and $V_-$.
\end{lemma}
\begin{proof}
  This is immediate from \eqref{eq:V0pm-Vpm-rel} and the
  fact that $V_\pm$ are closed subspaces of $V=H\times H$.
\end{proof}

\begin{remark}\label{rem:angularVpm}
  A result analogous to Corollary~\ref{cor:angular}
  holds for the subspaces $V_\pm$ of $V$ in the case that $T$ is strictly dichotomous,
  i.e.\ if $P_\pm\in L(V)$.
  In particular if $P_--Q_-$ is compact where $Q_-=\smallmat{I&0\\0&0}\in L(V)$
  and
  \[V_-=\range\pmat{I\\X_-},\qquad V_+=\range\pmat{Y_+\\I},\]
  then $X_-,Y_+\in L(H)$.
\end{remark}

\begin{theo}\label{theo:angularVpm}
  Suppose that $A$ has a compact resolvent and that $\max\{r,s\}<\frac12$.
  \begin{enumerate}
  \item If \eqref{eq:ac-theo} and \eqref{eq:ao-theo} hold, then
    $V_\pm=\range\smallmat{I\\X_\pm}$ where $X_\pm$ are the parts of $X_{0\pm}$
    in $H$. The operators $X_\pm$ are injective and satisfy $X_-,X_+^{-1}\in L(H)$.
  \item If $A$ is sectorial and $0\in\varrho(A)$, then
    $V_-=\range\smallmat{I\\X_-}$, $V_+=\range\smallmat{Y_+\\I}$
    where $X_-$ and $Y_+$ are the parts of $X_{0-}$ and $Y_{0+}$ in $H$,
   respectively, and $X_-,Y_+\in L(H)$.
  \end{enumerate}
\end{theo}
\begin{proof}
  The proof is analogous to the ones of Theorem~\ref{theo:angular-acao}
  and~\ref{theo:angular-stable}, where it is shown
  that $V_{0\pm}$ are angular subspaces.
  First note that $S$ and $T$ have a compact resolvent,
  see Lemma~\ref{lem:ham-cptresolv}.
  Second, since $\max\{r,s\}<\frac12$ and since $i\R\subset\varrho(T)$ by our general
  assumption in this section, Theorem~\ref{theo:Tdichot} in conjunction with
  Lemma~\ref{lem:pointgap} implies that $T$ is strictly dichotomous.
  Consequently the projections $P_\pm$ are bounded and satisfy
  \begin{equation*}
    P_{+}v-P_{-}v=\frac{1}{\pi i}\int_{-i\infty}^{i\infty\,\prime}
    (T-\lambda)^{-1}v\,d\lambda,
    \qquad v\in\mdef(T).
  \end{equation*}
  On the other hand, for $Q_\pm\in L(V)$ given by
  $Q_-=\smallmat{I&0\\0&0}$, $Q_+=\smallmat{0&0\\0&I}$ the identity
  \[Q_{+}v-Q_{-}v=\frac{1}{\pi i}\int_{\gamma_1}'(S-\lambda)^{-1}v\,d\lambda
  +Kv, \qquad v\in V,\]
  holds with some $K\in L(V)$, see Lemma~\ref{lem:SQpm} and
  Remark~\ref{rem:SQpm}.
  Consequently
  \begin{multline*}
    2(P_{-}-Q_{-})v=
    \frac{1}{\pi i}\int_{\gamma_1}
    (S-\lambda)^{-1}-(T-\lambda)^{-1}\,d\lambda \,v\\
    -\frac{1}{\pi i}\int_{-i\rho}^{i\rho}(T-\lambda)^{-1}\,d\lambda \,v
    +Kv
  \end{multline*}
  for $v\in V$, where we have used that in view of
  $\max\{r,s\}<\frac12$ and \eqref{eq:TSdiffest} all terms on the
  right-hand side yield bounded operators from $L(V)$.
  Since the resolvents of $S$ and $T$ are compact, we conclude that
  $P_--Q_-$ is compact too.
  The assertion now follows from
  Theorems~\ref{theo:angular-acao} and~\ref{theo:angular-stable},
  Lemma~\ref{lem:graphVpm} and Remark~\ref{rem:angularVpm}.
\end{proof}

\section{Symmetries of the angular operators}
\label{sec:sym}

The aim of this section is to derive symmetry properties
for the operators $X_{0\pm}$ and $X_\pm$.
We keep our general assumptions on the Hamiltonian:
$A$ is quasi-sectorial, $r+s<1$ and $i\R\subset\varrho(T_0)$.
Hence $T_0$ is bisectorial, strictly dichotomous and
the invariant subspaces are given by
\begin{equation*}
  V_{0\pm}=\range(P_{0\pm})=\ker L_{0\mp}
\end{equation*}
where $P_{0\pm}=T_0L_{0\pm}$, $L_{0\pm}\in L(V_0,V_1)$ and
\begin{equation}\label{eq:L0pm-sym}
  \braket{L_{0+}v}{w}=-\braket{v}{L_{0-}w}, \qquad v,w\in V_0,
\end{equation}
with the extended indefinite inner product
defined in \eqref{eq:extJinner},
see Lemma~\ref{lem:L0pm}.

For a subspace $U\subset V_1$ we consider its orthogonal complement
$U^\Jorth\subset V_0$  with respect to the extended inner product:
\[U^\Jorth=\set{w\in V_0}{\braket{v}{w}=0
  \text{ for all }v\in V_1}.\]
For $\tilde U\subset V_0$ the orthogonal complement $\tilde U^\Jorth\subset V_1$
is defined analogously.
Then, as in the usual Hilbert or Krein space setting, orthogonal complements are closed
and $U^{\Jorth\Jorth}=\overline{U}$.
Let $V_{1\pm}$ be the closure of $\range(L_{0\pm})$ in $V_1$,
\begin{equation}
  V_{1\pm}=\overline{\range(L_{0\pm})}^{V_1}.
\end{equation}

\begin{lemma}\label{lem:V1pm}
  The following identities hold:
  \begin{enumerate}
  \item \(V_{1\pm}^\Jorth=V_{0\pm}\),
  \item \(V_{1\pm}=V_{0\pm}\cap V_1\).
  \end{enumerate}
\end{lemma}
\begin{proof}
  \begin{enumerate}
  \item
    From \eqref{eq:L0pm-sym} we get
    \[V_{0\pm}=\ker L_{0\mp}\subset\range(L_{0\pm})^\Jorth=V_{1\pm}^\Jorth.\]
    If on the other hand $w\in V_{1\pm}^\Jorth$, then
    \(\braket{v}{L_{0\mp}w}=-\braket{L_{0\pm}v}{w}=0\)
    for all $v\in V_0$.
    Since the inner product is non-degenerate, this implies $L_{0\mp}w=0$
    and thus $w\in V_{0\pm}$.
  \item
    By Lemma~\ref{lem:Lpm-range} we have $\range(L_{0\pm})\subset V_{0\pm}$.
    By the continuity of the imbedding $V_1\hookrightarrow V_0$,
    the subspace
    $V_{0\pm}\cap V_1$ is closed in $V_1$, and hence
    the inclusion from left to right follows.
    For the reverse inclusion let $v\in V_{0\pm}\cap V_1$.
    Then
    \[\braket{w}{v}=0 \qquad\text{for all }w\in V_{1\pm}\]
    by (a).
    Since $T_0$ is densely defined and strictly dichotomous,
    Lemma~\ref{lem:Lpm-range} implies
    $\overline{\range(L_{0\pm})}^{V_0}=V_{0\pm}$.
    Hence $\overline{V_{1\pm}}^{V_0}=V_{0\pm}$ and therefore
    \[\braket{w}{v}=0 \qquad\text{for all }w\in V_{0\pm}.\]
    Consequently $v\in V_{0\pm}^\Jorth=V_{1\pm}^{\Jorth\Jorth}=V_{1\pm}$.
  \end{enumerate}
\end{proof}

Let $X_1:\mdef(X_1)\subset H_s\ads\to H_r$ be a densely defined operator.
We define its adjoint with respect to the scales of Hilbert spaces
$\{H_r\}$ and $\{H_s\ads\}$ as the operator
$X_1^*:\mdef(X_1^*)\subset H_{-r}\to H_{-s}\ads$ with maximal domain
such that
\begin{equation}\label{eq:adjop-scale}
  \iprods{X_1x}{y}{r}=\iprods{x}{X_1^*y}{s}\ads, \qquad
  x\in\mdef(X_1),\,y\in\mdef(X_1^*).
\end{equation}
Then $X_1^*$ is uniquely determined and closed.

\begin{lemma}\label{lem:X10-prop}
  If $V_{0-}=\range\smallmat{I\\X_{0-}}$ with a closed operator
  \[X_{0-}:\mdef(X_{0-})\subset H_{-r}\to H_{-s}\ads,\]
  then also $V_{1-}=\range\smallmat{I\\X_{1-}}$ with a closed operator
  \[X_{1-}:\mdef(X_{1-})\subset H_s\ads\to H_r.\]
  In this case:
  \begin{enumerate}
  \item $\mdef(X_{1-})=\set{x\in\mdef(X_{0-})\cap H_s\ads}{X_{0-}x\in H_r}$, i.e.,
    $X_{1-}$ is the part of $X_{0-}$ in the space of operators from
    $H_s\ads$ to $H_r$;
  \item $X_{1-}$ and $X_{0-}$ are densely defined and $X_{1-}^*=X_{0-}$;
  \item the set
    $\set{x\in\mdef(X_{0-})\cap H_{1-r}\ads}{X_{0-}x\in H_{1-s}}$ is a core
    for $X_{1-}$ and $X_{0-}$.
  \end{enumerate}
  Analogous statements hold for the spaces $V_{0+},V_{1+}$ and the operators
  $X_{0+},X_{1+}$.
\end{lemma}
\begin{proof}
  The inclusion $V_{1-}\subset V_{0-}$ implies that
  if $V_{0-}$ is a graph, then so is $V_{1-}$ and that $X_{1-}$ is a restriction
  of $X_{0-}$.
  $X_{1-}$ is closed since $V_{1-}$ is closed in $V_1=H_s\ads\times H_r$.
  (a) is now immediate from $V_{1-}=V_{0-}\cap V_1$.

  To show (b),
  suppose $x\in H_{-r}$, $y\in H_{-s}\ads$ are such that
  \begin{equation}\label{eq:X1-adj}
    \iprods{X_{1-}u}{x}{r}=\iprods{u}{y}{s}\ads\qquad
  \text{for all }u\in\mdef(X_{1-}).
  \end{equation}
  Then
  \[\Braket{\pmat{u\\X_{1-}u}}{\pmat{x\\y}}=0, \qquad u\in\mdef(X_{1-}),\]
  i.e., $\smallmat{x\\y}\in V_{1-}^\Jorth=V_{0-}$
  and thus $x\in\mdef(X_{0-})$, $X_{0-}x=y$.
  This implies that $\mdef(X_{1-})$ is dense in $H_s\ads$.
  Indeed if $y\in H_{-s}\ads$ with $\iprods{u}{y}{s}=0$ for all
  $u\in\mdef(X_{1-})$, then \eqref{eq:X1-adj} holds with $x=0$ and
  it follows that $y=0$.
  On the other hand $V_{1-}^\Jorth=V_{0-}$ implies
  \[i\iprods{u}{X_{0-}x}{s}\ads-i\iprods{X_{1-}u}{x}{r}=
  \Braket{\pmat{u\\X_{1-}u}}{\pmat{x\\X_{0-}x}}=0\]
  for all $u\in\mdef(X_{1-})$, $x\in\mdef(X_{0-})$
  and therefore $X_{0-}\subset X_{1-}^*$.
  Moreover if $x\in \mdef(X_{1-}^*)$ and $y=X_{1-}^*x$, then $x,\,y$  satisfy
  \eqref{eq:X1-adj} and we obtain $x\in\mdef(X_{0-})$.
  Consequently $X_{0-}=X_{1-}^*$.
  Finally $X_{0-}$ is densely defined since $\mdef(X_{1-})$ is dense
  in $H_s\ads$
  and the imbedding $H_s\ads\hookrightarrow H_{-r}$ is continuous
  and dense.

  Finally (c) follows from the equivalence
  \begin{align*}
    u\in\mdef(X_{0-})\cap H_{1-r}\ads \,\wedge\, X_{0-}u\in H_{1-s}
    \quad\Longleftrightarrow\quad \pmat{u\\X_{0-}u}\in V_{0-}\cap\mdef(T_0)
  \end{align*}
  in conjunction with
  $\range(L_{0-})=V_{0-}\cap\mdef(T_0)$,
  $V_{0-}=\overline{\range(L_{0-})}^{V_0}$,
  see Lemma~\ref{lem:Lpm-range}, and
  $V_{1-}=\overline{\range(L_{0-})}^{V_1}$.
\end{proof}

\begin{remark}
  The previous lemma implies $X_{1\pm}\subset X_{0\pm}=X_{1\pm}^*$.
  From this identity and \eqref{eq:adjop-scale} we obtain
  \[\iprod{X_{1\pm}x}{y}=\iprod{x}{X_{1\pm}y},\qquad x,y\in\mdef(X_{1\pm}).\]
  Consequently, if we consider $X_{1\pm}$ as an unbounded operator
  on $H$, then it is densely defined and symmetric and hence closable.
  The corresponding closure will be determined in Lemma~\ref{lem:XM-prop}.
\end{remark}

Now we turn to the symmetry properties of the operators $X_\pm$.
To this end, we look at the subspaces
\begin{equation}\label{eq:Mpm}
  M_\pm=\overline{\range(L_\pm)}^V
\end{equation}
of $V$.
By Lemma~\ref{lem:Lpm-range} we have $M_\pm\subset V_\pm$
and this inclusion may be strict.
The next lemma shows that $M_\pm^\Jorth$
coincides with $V_\pm$.
Note here that since $M_\pm\subset V$,
$M_\pm^\Jorth$ is the orthogonal complement with respect to the
inner product $\braket{\cdot}{\cdot}$ \emph{in $V$},
i.e.\ $M_\pm^\Jorth\subset V$
in the usual Krein space sense.

\begin{lemma}\label{lem:Mpm}
  The following identities hold:
  \begin{enumerate}
  \item $V_{1\pm}\subset M_\pm$ and $\overline{V_{1\pm}}^V=M_\pm$;
  \item $M_\pm^\Jorth=V_\pm$.
  \end{enumerate}
\end{lemma}

\begin{proof}
  \begin{enumerate}
  \item
    Since $\mdef(T_0)$ is dense in $V_0$ and $L_{0\pm}\in L(V_0,V_1)$,
    we have
    \[V_{1\pm}=\overline{\range(L_{0\pm})}^{V_1}
    \subset \overline{L_{0\pm}(\mdef(T_0))}^{V_1}
    \subset \overline{L_{0\pm}(\mdef(T_0))}^{V}\subset
    \overline{L_{0\pm}(V)}^V=M_\pm.\]
    On the other hand $\range(L_\pm)\subset \range(L_{0\pm})\subset V_{1\pm}$,
    which implies $M_\pm\subset\overline{V_{1\pm}}^V$
    and thus equality.
  \item
    Lemma~\ref{lem:L0pm} implies
    $\braket{L_+v}{w}=-\braket{v}{L_-w}$ for all $v,w\in V$.
    Using this and the definitions of $V_\pm$ and $M_\pm$,
    the proof is completely analogous to Lemma~\ref{lem:V1pm}(a).
  \end{enumerate}  
\end{proof}


\begin{lemma}\label{lem:XM-prop}
  Suppose $V_{0-}$ is a graph subspace $V_{0-}=\range\smallmat{I\\X_{0-}}$.
  Then $V_-=\range\smallmat{I\\X_-}$ and $M_-=\range\smallmat{I\\X_{M-}}$
  where $X_-,X_{M-}$ are closed operators on $H$.
  Moreover
  \begin{enumerate}
  \item $X_{M-}\subset X_-$,
  \item $X_-$ is the part of $X_{0-}$ in $H$,
  \item $X_{M-}$ is the closure of $X_{1-}$ when considered as an operator on $H$,
  \item $\bigset{x\in\mdef(X_{0-})\cap H_{1-r}\ads}{X_{0-}x\in H_{1-s}}$
    is a core for $X_{M-}$,
  \item $X_{M-}$ and $X_-$ are densely defined and $X_{M-}^*=X_-$.
    In particular $X_{M-}$ is symmetric.
  \end{enumerate}
  Again, analogous statements hold for $V_{0+}$, $V_+$ and $M_+$ and the respective
  operators.
\end{lemma}
\begin{proof}
  The first assertions up to (c) follow readily from
  $M_-\subset V_-\subset V_{0-}$,
  $V_-=V_{0-}\cap V$, $\overline{V_{1-}}^V=M_-$ and
  the closedness of $M_-$ and $V_-$ in $V$.
  (d) is a consequence of (c) and Lemma~\ref{lem:X10-prop}(c),
  and (e) follows from $M_-^\Jorth=V_-$ in an analogous way to the proof
  of Lemma~\ref{lem:X10-prop}(b).
\end{proof}

\begin{lemma}\label{lem:XM-nonneg}
  The symmetric operators $X_{M-}$ and $X_{M+}$ are nonnegative
  and nonpositive, respectively.
\end{lemma}

\begin{proof}
  Here we employ the indefinite inner product $\braket{\cdot}{\cdot}_\sim$ defined
  in \eqref{eq:Jinner}.
  Observe that $X_{M-}$ is nonnegative, i.e., $\iprod{X_{M-}x}{x}\geq0$
  for all $x\in\mdef(X_{M-})$,
  if and only if $\braket{v}{v}_\sim\geq0$ for all $v\in M_-$.
  Likewise $\iprod{X_{M+}x}{x}\leq0$ for all $x\in\mdef(X_{M+})$
  if and only if $\braket{v}{v}_\sim\leq0$ for all $v\in M_+$.
  Consider first $v\in\mdef(T)$.
  Using \eqref{eq:dichotint2} and Lemma \ref{lem:Hamsym}, we calculate
  \begin{align*}
    \Real\braket{P_+v-P_-v}{v}_\sim
    &=\frac{1}{\pi}\int_{-\infty}^{\infty\,\prime}\Real\braket{(T-it)^{-1}v}{v}_\sim\,dt\\
    &=\frac{1}{\pi}\int_{-\infty}^{\infty\,\prime}
    \Real\braket{(T-it)^{-1}v}{(T-it)(T-it)^{-1}v}_\sim\,dt\\
    &=\frac{1}{\pi}\int_{-\infty}^{\infty\,\prime}
    \Real\braket{T(T-it)^{-1}v}{(T-it)^{-1}v}_\sim\,dt\leq0.
  \end{align*}
  If now $v\in\mdef(T)\cap V_-$ then $P_+v-P_-v=-v$ and hence
  $\braket{v}{v}_\sim\geq0$.
  Since $\mdef(T)\cap V_-$ is dense in $M_-$ by Lemma~\ref{lem:Lpm-range},
  we conclude that $\braket{v}{v}_\sim\geq0$ for $v\in M_-$.
  Similarly for $v\in\mdef(T)\cap V_+$ we obtain $P_+v-P_-v=v$
  and thus $\braket{v}{v}_\sim\leq0$ for all $v\in M_+$.
\end{proof}

\begin{coroll}
  If $\max\{r,s\}<\frac12$, then $X_{M\pm}=X_\pm$.
  The operator $X_-$ is selfadjoint and nonnegative, $X_+$ is selfadjoint
  and nonpositive.
\end{coroll}
\begin{proof}
  The assumption implies that $T$ is strictly dichotomous. Then
  $M_\pm=V_\pm$ by Lemma~\ref{lem:Lpm-range}
  and hence $X_{M\pm}=X_\pm$.
\end{proof}

\section{The Riccati equation}\label{sec:ricc}

We keep the general assumptions of the previous section.

\begin{lemma}
  Suppose $X_0\in L(H_{-r},H_{-s}\ads)$ is such that its graph subspace
  $U=\range\smallmat{I\\X_0}$ is $T_0$- and $(T_0-\lambda)^{-1}$-invariant.
  Consider
  the isomorphism $\varphi:H_{-r}\to U$, $x\mapsto\smallmat{x\\X_0x}$.
  Then
  \begin{enumerate}
  \item $X_0(H_{1-r}\ads)\subset H_{1-s}$;
  \item $(A-BB^*X_0)x=\varphi^{-1}T_0|_U\varphi x$
    \; for all $x\in H_{1-r}\ads$;    
  \item $A^*X_0x+X_0Ax-X_0BB^*X_0x+C^*Cx=0$ \; for all $x\in H_{1-r}\ads$.
  \end{enumerate}
\end{lemma}

\begin{proof}
  First note that $\varphi$ is indeed an isomorphism
  between $H_{-r}$ and $U$ since $X_0$ is bounded.
  The inverse is $\varphi^{-1}=\proj_1|_U$ where
  \[\proj_1:V_0=H_{-r}\times H_{-s}\ads\to H_{-r}\]
  denotes the projection onto the first component.
  Recall the decomposition $T_0=S_0+R$ from \eqref{eq:T0decomp}
  and consider the two operators
  $F=\varphi^{-1}T_0|_U\varphi$ and $A_0=\proj_1 S_0\varphi$,
  both understood as unbounded operators on $H_{-r}$.
  Since $\mdef(T_0)=\mdef(S_0)=H_{1-r}\ads\times H_{1-s}$,
  their domains are
  \[\mdef(F)=\mdef(A_0)=\set{x\in H_{1-r}\ads}{X_0x\in H_{1-s}}.\]
  Moreover
  \[A_0x=Ax\qquad\text{for }x\in\mdef(A_0),\]
  i.e., $A_0$ is a restriction of $A$ when $A$ is considered as an operator on $H_{-r}$
  with $\mdef(A)=H_{1-r}\ads$.
  Since $\varphi$ is an isomorphism we get  $\varrho(F)=\varrho(T_0|_U)$.
  Also $\varrho(T_0)\subset\varrho(T_0|_U)$ by the invariance of $U$.
  Therefore $i\R\subset\varrho(F)$. For $t\in\R$ we compute
  \begin{align*}
    (A_0-F)(F-it)^{-1}
    &=(\proj_1 S_0\varphi-\varphi^{-1}T_0\varphi)(\varphi^{-1}T_0\varphi-it)^{-1}\\
    &=\proj_1(S_0-T_0)\varphi\varphi^{-1}(T_0-it)^{-1}\varphi
    =-\proj_1 R(T_0-it)^{-1}\varphi.
  \end{align*}
  From \eqref{eq:T0est-V0V1} in the proof of Lemma~\ref{lem:L0pm} we know that
  $\|(T_0-it)^{-1}\|_{L(V_0,V_1)}\to0$ as $t\to\infty$, and we conclude that
  $\|(A_0-F)(F-it)^{-1}\|<1$ for $t>0$ sufficiently large.
  Now
  \[A_0-it=F-it+A_0-F=\left(I+(A_0-F)(F-it)^{-1}\right)(F-it),\]
  which implies that $it\in\varrho(A_0)$.
  Since also $it\in\varrho(A)$ for large $t$ and $A_0\subset A$, it follows that
  in fact
  \[\mdef(A_0)=\mdef(A)=H_{1-r}\ads.\]
  Consequently
  $X_0(H_{1-r}\ads)\subset H_{1-s}$.
  Since $Fx=Ax-BB^*X_0x$ for $x\in\mdef(F)=\mdef(A_0)$,
  (b) is now clear.
  To show (c) let $x\in H_{1-r}\ads$. Then $X_0x\in H_{1-s}$ and $\varphi x\in\mdef(T_0)$.
  By the invariance of $U$ there exists $y\in H_{1-r}\ads$ such that
  $T_0\varphi x=\varphi y$, i.e.,
  \[\pmat{A&-BB^*\\-C^*C&-A^*}\pmat{x\\X_0x}=
  \pmat{y\\X_0y}\]
  and thus
  \[X_0Ax-X_0BB^*X_0x=X_0(Ax-BB^*X_0x)=X_0y=-C^*Cx-A^*X_0x.\]
\end{proof}

\begin{coroll}\label{cor:closed-loop}
  If $V_{0-}=\range\smallmat{I\\X_{0-}}$ with a bounded operator
  $X_{0-}\in L(H_{-r},H_{-s}\ads)$, then
  $X_{0-}(H_{1-r}\ads)\subset H_{1-s}$,
  the Riccati equation
  \[A^*X_{0-}x+X_{0-}Ax-X_{0-}BB^*X_{0-}x+C^*Cx=0,\qquad x\in H_{1-r}\ads,\]
  holds, and 
  $A-BB^*X_{0-}$
  considered as an unbounded operator on $H_{-r}$
  is sectorial 
  with spectrum $\sigma(A-BB^*X_{0-})\subset \C_-$.
  In particular, it generates an exponentially stable analytic semigroup
  on $H_{-r}$.
\end{coroll}
\begin{proof}
  $A-BB^*X_{0-}$ is similar  to $T_0|_{V_{0-}}$ via the isomorphism $\varphi$
  from the previous lemma, $\sigma(T_0|_{V_{0-}})\subset\C_-$, and $T_0|_{V_{0-}}$
  is sectorial by \cite[Theorem~5.6]{winklmeier-wyss15}.
\end{proof}

\begin{remark}
  If $X_{0-}\in L(H_{-r},H_{-s}\ads)$ and hence
  $X_{0-}(H_{1-r}\ads)\subset H_{1-s}$,
  Lemma~\ref{lem:X10-prop} and~\ref{lem:XM-prop} imply
  that $H_{1-r}\ads\subset\mdef(X_{1-})\subset\mdef(X_{-})$.
  Since the operator $A-BB^*X_{0-}$ considered on $H_{-r}$ has domain $H_{1-r}\ads$
  we find that
  \[A-BB^*X_{0-}=A-BB^*X_{-}=A-BB^*X_{1-}.\]
  Hence the Riccati equation can be written as
  \[A^*X_{1-}x+X_{0-}Ax-X_{0-}BB^*X_{1-}x+C^*Cx=0,\qquad x\in H_{1-r}\ads,\]
  or in weak form, using $X_{0-}=X_{1-}^*$, as
  \begin{multline*}
    \iprods{X_{1-}x}{Ay}{r}+\iprodsr{Ax}{X_{1-}y}{r}
    -\iprod{B^*X_{1-}x}{B^*X_{1-}y}_U\\
    +\iprod{Cx}{Cy}_Y=0,
    \quad x,y\in H_{1-r}\ads.
  \end{multline*}
  Of course, in both Riccati equations $X_{1-}$ may be replaced
  by one of its extensions $X_{M-}$ and $X_-$.
\end{remark}
\begin{remark}
  For $X_{0-}\in L(H_{-r},H_{-s}\ads)$
  Corollary~\ref{cor:closed-loop} yields that $A-BB^*X_-$
  is sectorial \emph{when considered as an operator in $H_{-r}$.}
  On the other hand,
  we can consider the part of $A-BB^*X_-$ in $H$, which we denote
  by $(A-BB^*X_-)|_H$.
  Then $(A-BB^*X_-)|_H$ is almost sectorial:
  First note that
  \[\sigma\bigl((A-BB^*X_-)|_H\bigr)\subset\sigma(A-BB^*X_-).\]
  From $A-BB^*X_-=\varphi^{-1} T_0|_{V_{0-}}\varphi$ we obtain
  \[\|(A-BB^*X_--\lambda)^{-1}\|_{L(H_{-r},H)}\leq
  \|(T_0|_{V_{0-}}-\lambda)^{-1}\|_{L(V_{0-},V)}\|\varphi\|,\]
  and \eqref{eq:T0resolvest-V} in conjunction with
  $i\R\subset\varrho(A-BB^*X_-)$ implies
  \[\|(A-BB^*X_--\lambda)^{-1}\|_{L(H_{-r},H)}\leq \frac{c_0}{|\lambda|^{1-\max\{r,s\}}}
  \quad\text{for } \lambda\in i\R\setminus\{0\},\]
  with some constant $c_0>0$.
  Moreover since $\|(T_0|_{V_{0-}}-\lambda)^{-1}\|_{L(V_0)}$ is bounded on $\C_+$,
  $\|(T_0|_{V_{0-}}-\lambda)^{-1}\|_{L(V_0,\mdef(T_0))}$ does not grow faster than
  $|\lambda|$ on $\C_+$, where $\mdef(T_0)$ is equipped with the graph norm.
  As the imbedding $\mdef(T_0)\hookrightarrow V$ is continuous,
  $\|(A-BB^*X_--\lambda)^{-1}\|_{L(H_{-r},H)}$ does not grow faster than $|\lambda|$ on $\C_+$ too.
  The Phragm\'en-Lindel\"of theorem then implies that
  \[\|(A-BB^*X_--\lambda)^{-1}\|_{L(H_{-r},H)}\leq \frac{c_0}{|\lambda|^{1-\max\{r,s\}}}
  \quad\text{for }\lambda\in\overline{\C_+}\setminus\{0\}\]
  and hence $(A-BB^*X_-)|_H$ is almost sectorial,
  see \cite[\S5]{winklmeier-wyss15}.

  Now suppose in addition that $\max\{r,s\}<\frac12$ and that $X_-\in L(H)$,
  e.g.\ as a consequence of
  Theorem~\ref{theo:angularVpm}. Then
  \[(A-BB^*X_-)|_H=\varphi|_H^{-1}T|_{V_-}\varphi|_H\]
  where $\varphi|_H:H\to V_-$, $x\mapsto\smallmat{x\\X_-x}$ is an isomorphism.
  Since $T$ is bisectorial by Theorem~\ref{theo:Tdichot},
  $T|_{V_-}$ is sectorial by  \cite[Theorem~5.6]{winklmeier-wyss15},
  and hence $(A-BB^*X_-)|_H$ is sectorial too.
\end{remark}

\bibliographystyle{cwyss}
\bibliography{biblist}

\end{document}